\documentclass[11pt]{amsart}
\usepackage[T1]{fontenc}
\usepackage{lmodern,microtype,mathtools,amssymb,mathrsfs}
\usepackage{booktabs,array,enumitem,needspace}
\usepackage[margin=1.10in]{geometry}
\usepackage[hidelinks]{hyperref}
\hypersetup{pdftitle={Tame Discrete Sets on Affine Semisimple Homogeneous Spaces},pdfauthor={Alexander Dvorsky},pdfsubject={rrhom v1r}}
\allowdisplaybreaks[1]
\newcommand{\C}{\mathbb C}
\newcommand{\N}{\mathbb N}

\newcommand{\OO}{\mathscr O}
\newcommand{\Aut}{\operatorname{Aut}}
\newcommand{\Autfl}{\operatorname{Aut}_{\mathrm{fl}}}
\newcommand{\Autvfl}{\operatorname{Aut}_{\omega,\mathrm{fl}}}
\newcommand{\salg}{\mathfrak s}
\newcommand{\Lie}{\operatorname{Lie}}
\newcommand{\Ad}{\operatorname{Ad}}
\newcommand{\Span}{\operatorname{span}}
\newcommand{\diag}{\operatorname{diag}}

\newcommand{\rk}{\operatorname{rank}}
\newcommand{\SL}{\mathrm{SL}}
\newcommand{\GL}{\mathrm{GL}}
\newcommand{\SO}{\mathrm{SO}}
\newcommand{\Ort}{\mathrm O}
\newcommand{\Sp}{\mathrm{Sp}}
\newcommand{\Spin}{\mathrm{Spin}}
\newcommand{\SU}{\mathrm{SU}}
\newcommand{\Sym}{\operatorname{Sym}}
\newcommand{\Pf}{\operatorname{Pf}}
\newcommand{\id}{\operatorname{id}}
\newcommand{\g}{\mathfrak g}
\newcommand{\h}{\mathfrak h}
\newcommand{\tlie}{\mathfrak t}
\newcommand{\ulie}{\mathfrak u}
\newcommand{\cT}{\mathcal T}
\newcommand{\diver}{\operatorname{div}}
\newtheorem{theorem}{Theorem}[section]
\newtheorem{maintheorem}{Theorem}

\newtheorem{proposition}[theorem]{Proposition}
\newtheorem{lemma}[theorem]{Lemma}
\newtheorem{corollary}[theorem]{Corollary}
\theoremstyle{definition}
\newtheorem{definition}[theorem]{Definition}

\theoremstyle{remark}
\newtheorem{remark}[theorem]{Remark}
\numberwithin{equation}{section}
\title[Tame sets on semisimple homogeneous spaces]{Tame Discrete Sets on Affine Semisimple Homogeneous Spaces}
\author{Alexander Dvorsky}
\address{Department of Mathematics, University of Miami, Coral Gables, Florida 33146, USA}
\email{dvorsky@math.miami.edu}
\date{October 4, 2026}
\begin{document}
\begin{abstract}
Winkelmann conjectures that every smooth flexible complex affine variety of dimension at least two is a Rosay--Rudin space. We verify this for every positive-dimensional quotient $G/H$ with $G$ a connected complex semisimple algebraic group and $H$ a closed connected reductive subgroup. Weak and strong tameness coincide, every injection between tame discrete sets extends to a holomorphic automorphism, and complements of tame, finite, or empty sets are Oka. Every sufficiently sparse enumerated sequence can be sent to a fixed sequence by a composition of $n$ complete holomorphic flow maps, with $n\le6$; injective self-maps of the fixed sequence admit such a realization with $n\le2$. These flows preserve an invariant algebraic volume form. The proof uses entire interpolation and regular functions invariant under two solvable subgroups to construct the automorphisms. If $G$ is simple of rank at least two and $G/H$ is spherical and admits a closed equivariant embedding into an irreducible module, the normalization bound improves to $n\le4$. We also prove that $\SL_2/N(T)$, where $T$ is a maximal torus, is an RR-space. This quotient is the complement of a smooth conic in $\mathbf P^2$.
\end{abstract}
\maketitle

\section{Introduction}

On an affine homogeneous space of a complex semisimple group, holomorphic automorphisms can send any finite ordered set of distinct points to any other such set of the same size. For infinite closed discrete sets, the corresponding extension problem requires a condition on the sets. We prove that weak tameness suffices: on the spaces considered here, every injection between weakly tame sets extends to a holomorphic automorphism. Weak tameness asks only that the points can be moved beyond arbitrarily assigned exhaustion levels.

Rosay and Rudin introduced tame discrete sets in $\C^n$, $n\ge2$, by requiring an automorphism that sends the set onto $\N\times\{0\}^{n-1}$ \cite{RosayRudin}. Two intrinsic definitions extend this notion to other complex manifolds. Weak tameness was developed by Winkelmann, and strong tameness was introduced by Andrist and Ugolini \cite{WinkelmannStein,AndristUgolini}.

All algebraic groups and varieties are over $\C$. A semisimple (respectively, simple) group is connected with semisimple (respectively, nonabelian simple) Lie algebra. We write $\N=\{1,2,\ldots\}$ and $\Aut(X)$ for the holomorphic automorphism group. By a discrete set we mean a closed discrete subset of the ambient manifold.

A \emph{proper exhaustion} on a complex manifold $X$ is a continuous function $\rho:X\to[0,\infty)$ such that $\{x\in X:\rho(x)\le R\}$ is compact for every $R\ge0$.

\begin{definition}\label{def:tameness}
Let $\Gamma\subseteq\Aut(X)$ be a subgroup and let $D\subset X$ be an infinite closed discrete set. The set $D$ is \emph{weakly $\Gamma$-tame} if, for every proper exhaustion $\rho:X\to[0,\infty)$ and every function $b:D\to[0,\infty)$, there is $\Phi\in\Gamma$ such that
\[
 \rho(\Phi(x))\ge b(x)\qquad(x\in D).
\]
It is \emph{strongly $\Gamma$-tame} if every injection $D\to D$ extends to an element of $\Gamma$. When $\Gamma=\Aut(X)$, we omit the group.
\end{definition}

Strong tameness implies weak tameness: one can choose a self-injection whose images lie beyond the assigned exhaustion levels. The converse holds in $\C^n$ but fails on some Stein manifolds \cite[Section~5]{WinkelmannStein}.

\subsection{The main RR theorem}
Winkelmann's definition of a Rosay--Rudin space requires a family of infinite closed discrete sets with the following properties: injections between its members extend to automorphisms; the family is invariant under automorphisms; and every sufficiently sparse infinite set belongs to it. Sparseness is measured by a \emph{threshold sequence}: for some proper exhaustion $\rho$ and positive numbers $r_n\uparrow\infty$, every infinite subset $D\subset X$ satisfies the implication
\begin{equation}\label{eq:intro-threshold}
 \#\{x\in D:\rho(x)\le r_n\}\le n\quad(n\ge1)
 \quad\Longrightarrow\quad D\text{ belongs to the family}.
\end{equation}
The definition also requires a nonzero complete holomorphic vector field annihilating a nonconstant holomorphic function \cite[Definition~2.1]{WinkelmannRR}. Section~\ref{sec:tameness} verifies these four requirements for the family of weakly tame sets.

\begin{maintheorem}\label{thm:RR}
Let $G$ be a connected complex semisimple algebraic group and let $H\subset G$ be a closed connected reductive subgroup. If $X=G/H$ has positive dimension, then $X$ is a Rosay--Rudin space. In particular, weak and strong tameness coincide, every injection between tame infinite closed discrete sets extends to a holomorphic automorphism, and $X$ admits a threshold sequence.
\end{maintheorem}

A smooth affine variety is \emph{flexible} if its tangent spaces are spanned by the velocity vectors of algebraic additive-group actions. In dimension at least two, flexibility is equivalent to transitivity of the group generated by these actions and implies its transitivity on finite ordered configurations of distinct points \cite[Theorem~0.1]{AFKKZ}. The spaces in Theorem~\ref{thm:RR} are affine by reductivity of $H$. They are flexible because the nilpotent elements span the semisimple Lie algebra, whose action fields span every tangent space. Thus Theorem~\ref{thm:RR} verifies for these homogeneous spaces Winkelmann's conjecture that every smooth flexible complex affine variety of dimension at least two is an RR-space \cite[Section~2.6]{WinkelmannRR}.

Andrist and Arzhantsev recently proved generic infinite transitivity for many groups generated by two root subgroups of $\Aut(\C^2)$: their diagonal actions on $(\C^2)^m$ have an open orbit for every $m$ \cite{AndristArzhantsev}. Theorem~\ref{thm:RR} concerns interpolation on infinite closed discrete sets under weak tameness.

Winkelmann also conjectures that the density property and the RR property are equivalent for Stein manifolds \cite[Section~2.6]{WinkelmannRR}. The density property means that the Lie algebra generated by complete holomorphic vector fields is dense in the space of all holomorphic vector fields, for uniform convergence on compact sets \cite{Varolin}. Donzelli, Dvorsky, and Kaliman proved algebraic density results for affine homogeneous spaces \cite{DonzelliDvorskyKaliman}; Kaliman and Kutzschebauch established the algebraic density property for every connected affine homogeneous space other than the line and algebraic tori \cite{KalimanKutzschebauch}. Their results imply the density property for the spaces in Theorem~\ref{thm:RR}. Recent criteria of Andrist, Draisma, Freudenburg, Huang, and Kutzschebauch use $\mathfrak{sl}_2$-triples of locally nilpotent derivations to establish algebraic density \cite{AndristDraismaEtAl}. Our proof constructs the automorphisms needed for the RR property directly.

Andrist and Ugolini proved that a connected positive-dimensional affine homogeneous space with reductive stabilizer contains a strongly tame set unless the space is isomorphic to $\C$ or $\C^*$ \cite[Theorem~1.5]{AndristUgoliniHom}. On a Stein manifold with the density property, they also proved that any two strongly tame sets are equivalent under an automorphism \cite[Proposition~2.4]{AndristUgolini}. Here we show that every weakly tame set can be sent to one fixed sequence. We also obtain a fixed threshold sequence such that every infinite set satisfying its counting inequalities is tame.

The group case and the rank-one affine quadric were already known to be RR-spaces by Winkelmann's results on character-free algebraic groups and Danielewski surfaces \cite{WinkelmannGroups,WinkelmannRR}. The argument here uses functions invariant under two solvable subgroups and applies to arbitrary connected reductive $H$.

The companion paper \cite{DvorskyRegular} proves Hamiltonian tameness, relative Hamiltonian density, and symplectic jet interpolation on regular semisimple adjoint orbits of complex semisimple Lie algebras.

\subsection{Normalization by finitely many flows}
A \emph{flow map} is the time-one map of a complete autonomous holomorphic vector field, whose flow is defined for every complex time. Let $\Autfl(X)$ be the group of finite compositions of flow maps. For a nowhere-vanishing holomorphic volume form $\omega$, let $\Autvfl(X)$ be the subgroup generated by the flow maps of complete vector fields whose flows preserve $\omega$, or equivalently, whose $\omega$-divergence is zero.

\begin{maintheorem}\label{thm:main}
Under the hypotheses of Theorem~\ref{thm:RR}, there are distinct points $p_m$ forming a closed discrete set $D_0$, a proper exhaustion $\rho:X\to(0,\infty)$, and thresholds
\[
 R_1=0<R_2<R_3<\cdots,\qquad R_m\longrightarrow\infty,
\]
with the following properties.
\begin{enumerate}[label=\textup{(\alph*)}]
\item Let $(x_m)_{m\ge1}$ be a sequence of distinct points of $X$ satisfying
\[
 \rho(x_m)>R_m\qquad(m\ge1).
\]
Then there exist flow maps $\Phi_1,\ldots,\Phi_n$, with $1\le n\le6$, such that
\[
 (\Phi_n\circ\cdots\circ\Phi_1)(x_m)=p_m\qquad(m\ge1).
\]
\item For every injection $\sigma:\N\to\N$, there exist flow maps $\Phi_1,\Phi_2$ such that
\[
 (\Phi_2\circ\Phi_1)(p_m)=p_{\sigma(m)}\qquad(m\ge1).
\]
\item For every subgroup $\Autfl(X)\subseteq\Gamma\subseteq\Aut(X)$, the following conditions on an infinite closed discrete set $D$ are equivalent: weak $\Gamma$-tameness, strong $\Gamma$-tameness, and equivalence to $D_0$ under $\Gamma$. For every weakly $\Gamma$-tame set $D$ and every enumeration $D=\{x_m:m\in\N\}$, there is $\Phi\in\Gamma$ such that $\Phi(x_m)=p_m$ for all $m$. Every injection between two weakly $\Gamma$-tame sets extends to an element of $\Gamma$.
\item There is a $G$-invariant algebraic volume form $\omega$. The vector fields generating the flow maps in \textup{(a)} and \textup{(b)} may be chosen with zero $\omega$-divergence, and \textup{(c)} holds for every subgroup $\Autvfl(X)\subseteq\Gamma\subseteq\Aut(X)$.
\end{enumerate}
\end{maintheorem}

Since $\rho>0$ and $R_1=0$, no restriction is imposed on $x_1$. The vector fields may depend on $(x_m)$, but the bound $n\le6$ is independent of the dimension and Lie type. For an arbitrary weakly tame set, one first chooses an automorphism imposing the growth inequalities. That preliminary automorphism is not included in this bound.

Part~\textup{(d)} also applies to
\[
 \Gamma=\Aut(X,\omega):=\{\Phi\in\Aut(X):\Phi^*\omega=\omega\}.
\]
Thus every injection between weakly $\Aut(X,\omega)$-tame sets extends to a volume-preserving automorphism. Parts~\textup{(c)} and~\textup{(d)} likewise identify weak and strong tameness with respect to $\Autfl(X)$ and $\Autvfl(X)$.

On a regular semisimple adjoint orbit $\mathcal O$ of dimension $2d$, let $\eta$ be the Kostant--Kirillov--Souriau form and put $\nu=\eta^d/d!$. The group $\operatorname{HAut}(\mathcal O)$ generated by complete Hamiltonian flows is contained in $\Aut_{\nu,\mathrm{fl}}(\mathcal O)$. Theorem~\ref{thm:main}(d) applies to groups containing $\Aut_{\nu,\mathrm{fl}}(\mathcal O)$; the companion paper realizes interpolation within $\operatorname{HAut}(\mathcal O)$. In dimension two these flow groups coincide, since $\nu=\eta$ and $\mathcal O$ is simply connected.

Taking $r_n=R_{n+1}$ gives the counting criterion \eqref{eq:intro-threshold}. Indeed, ordering a set satisfying the counting inequalities by nondecreasing $\rho$ gives $\rho(x_m)>R_m$ for $m\ge2$, so Theorem~\ref{thm:main}\textup{(a)} applies.

There is also an extension to disconnected stabilizers. Let $N$ be the kernel of the action of $G$ on $G/H$. Proposition~\ref{prop:disconnected} proves the RR property and Theorem~\ref{thm:main}\textup{(a)--(c)} for closed reductive $H$ with $\dim(H/N)>0$. The case where $H/N$ is finite and nontrivial remains outside the present argument. A $G$-invariant algebraic volume form exists precisely when every $h\in H$ acts with determinant one on $\g/\h$; here $\g=\Lie(G)$ and $\h=\Lie(H)$. The constructed flows preserve this form whenever it exists. The extension includes $\SL_2/N_{\SL_2}(T)\simeq\mathbf P^2\setminus\mathcal Q$, where $T$ is a maximal torus of $\SL_2$ and $\mathcal Q$ is a smooth conic. This surface is an RR-space but has no nowhere-vanishing holomorphic two-form.

\subsection{Two successive interpolations}
The construction can be seen on the quadric
\[
 \SL_2/T\simeq\left\{S=\begin{pmatrix}a&b\\b&c\end{pmatrix}:ac-b^2=1\right\},
 \qquad g\cdot S=gSg^{\mathsf T}.
\]
Upper and lower unipotent matrices act by
\[
 (a,b,c)\longmapsto(a+2tb+t^2c,b+tc,c),\qquad
 (a,b,c)\longmapsto(a,b+sa,c+2sb+s^2a).
\]
For a sufficiently sparse sequence, one element of $\SU_2$ makes the $c$-coordinates nonzero, pairwise distinct, and closed discrete. Write the new points as $(a_m,b_m,c_m)$. Fix square roots $\beta_m^2=m^2-1$, choose $d_m^2=mc_m-1$, and choose entire functions satisfying
\[
 F_+(c_m)=\frac{d_m-b_m}{c_m},\qquad
 F_-(m)=\frac{\beta_m-d_m}{m}.
\]
Substitute $t=F_+(c)$ in the upper action and then $s=F_-(a)$ in the lower action. The determinant equation gives
\[
 (a_m,b_m,c_m)\longmapsto(m,d_m,c_m)
 \longmapsto(m,\beta_m,m).
\]
The first shear preserves $c$ and makes the $a$-coordinate equal to $m$; the second preserves $a$ and is obtained by interpolation at these integers. Invariance of the coefficients makes both shears complete. Section~\ref{subsec:quadric} also gives the shears realizing injections of the fixed sequence into itself.

For the general construction, suppose that the action is effective, $\rk G\ge2$, and $\dim H>0$. Proposition~\ref{prop:model} constructs regular functions $q,r$, connected solvable subgroups $S_1,S_2$ fixing them, and points $p_m$ on one torus orbit with $q(p_m)=r(p_m)=m$. The functions transform under that torus by independent characters. The subgroups and points satisfy
\[
 \Lie(S_1)+\Lie(S_2)+\g_{p_m}=\g.
\]
Here $\g_{p_m}$ is the stabilizer Lie algebra. The map $(a,b)\mapsto abp_m$ therefore has surjective differential at $(1,1)$, and its image contains a nonempty Zariski-open set $\Omega_m$. The choice of subgroups uses positive-dimensional isotropy; the Lie algebra calculation is in Lemma~\ref{lem:transverse}.

Lemma~\ref{lem:prepare} sends a sufficiently sparse sequence to points $y_m\in\Omega_m$ with distinct divergent values $q(y_m)$. For each $m$, choose $a_m\in S_1$ and $b_m\in S_2$ with $y_m=a_mb_mp_m$. We realize
\begin{equation}\label{eq:intro-corrections}
 y_m=a_mb_mp_m
 \xrightarrow[\text{interpolate at }q(y_m)]{\ a_m^{-1}\ }
 b_mp_m
 \xrightarrow[\text{interpolate at }r(b_mp_m)=m]{\ b_m^{-1}\ }
 p_m.
\end{equation}
Factor $a_m^{-1}$ into a torus element and a unipotent element, and interpolate their logarithms at $q(y_m)$. Since both factors lie in $S_1$, their actions preserve $q$, and the interpolants define two complete flows. At the intermediate points, $r(b_mp_m)=r(p_m)=m$ because $S_2$ fixes $r$. Interpolating the factors of $b_m^{-1}$ at these integers gives two further flows.

To obtain the points $y_m$, Lemma~\ref{lem:prepare} uses one compact group element and one unipotent shear. In a closed equivariant embedding, choose a tuple $Q$ of linear functions fixed by the unipotent subgroup. A uniform measure estimate gives one compact element $k$ for which the values $Q(kx_m)$ are distinct and divergent and the points $kx_m$ meet the required open-set conditions. Interpolating the shear times at $Q(kx_m)$ makes the values $q(y_m)$ distinct with moduli tending to infinity. The compact action is itself a complete flow map, giving $n\le6$ in Theorem~\ref{thm:main}(a).

For a general space in Theorem~\ref{thm:RR}, remove the action kernel and write $X=\bar G/\bar H$. If $\rk\bar G\ge2$ and $\dim\bar H>0$, Proposition~\ref{prop:principal} applies. If $\bar H$ is trivial, use the left-right action on $\bar G$, whose diagonal stabilizer has positive dimension. The remaining case is $\SL_2/T$. In each case, injective self-maps of the fixed sequence extend by two flows. Section~\ref{sec:tameness} deduces the RR axioms from these extensions and the normalization of sparse sequences.

\subsection{Examples and consequences}
The quotient
\[
 \SL_3(\C)\big/\{\diag(t,t^{-1},1):t\in\C^*\}
\]
has dimension seven, whereas a Borel subgroup has dimension five, so it is not spherical. Section~\ref{sec:cstar} gives explicit formulas for the normalizing maps.

Suppose now that $G$ is simple of rank at least two and $G/H$ is spherical, meaning that a Borel subgroup has an open orbit. If $G/H$ admits a closed equivariant embedding into an irreducible module, Proposition~\ref{prop:position} applies to a single highest-weight linear function. For a sufficiently sparse sequence, one compact action puts its points in the open Borel orbit with distinct, discrete values of that function. A unipotent flow moves the points into one torus orbit, and two torus flows send them to the fixed sequence. Thus the bound in Theorem~\ref{thm:main}\textup{(a)} improves to $n\le4$.

\begin{corollary}\label{cor:spherical}
Under the hypotheses of Theorem~\ref{thm:main}, suppose $G$ is simple of rank at least two, $H$ is spherical, and $N_G(H)/H$ is finite cyclic. The set $D_0$, its enumeration, the exhaustion, and the thresholds can be chosen so that all conclusions of that theorem hold, with $n\le4$ in \textup{(a)}.
\end{corollary}

For reductive spherical $H$, Losev's theorem says that $G/H$ admits a closed equivariant embedding into an irreducible $G$-module if and only if $N_G(H)/H$ is finite cyclic \cite[Theorem~1.4 and Example~5.1]{Losev}. The construction in Section~\ref{sec:spherical} also applies to disconnected reductive $H$ when $G$ is simple of rank at least two, $G/H$ is spherical, and it admits a closed equivariant embedding into an irreducible module. Its flows preserve a $G$-invariant algebraic volume form whenever such a form exists. Section~\ref{sec:matrices} writes the flow maps explicitly for $\SL_3/\SO_3$, realized as the symmetric $3\times3$ matrices of determinant one.

Winkelmann's results on RR-spaces give two further consequences. The first concerns complements:
\begin{corollary}\label{cor:oka}
Under the hypotheses of Theorem~\ref{thm:RR}, if $D\subset X$ is tame, finite, or empty, then $X\setminus D$ is Oka.
\end{corollary}

The second recovers the affine-line description of tame sets. Fix a closed algebraic curve $C\simeq\mathbb A^1$ in $X$ and an infinite closed discrete set $D_C\subset C$. An infinite closed discrete set $D\subset X$ is tame if and only if some automorphism sends it onto $D_C$; see Corollary~\ref{cor:line-normal-form}. Such curves exist: at least one root subgroup acts nontrivially, and each of its nontrivial orbits is closed in the affine variety and isomorphic to $\mathbb A^1$.

Sections~\ref{sec:flows}--\ref{sec:tameness} prove Theorems~\ref{thm:RR} and~\ref{thm:main}. Section~\ref{sec:spherical} proves Corollary~\ref{cor:spherical}, and Section~\ref{sec:matrices} gives the explicit normalizing maps for $\SL_3/\SO_3$. Section~\ref{sec:examples} computes $N_G(H)/H$ for several spherical pairs and determines which satisfy the finite cyclic hypothesis of Corollary~\ref{cor:spherical}.

\section{Complete flows from invariant functions}\label{sec:flows}

Entire interpolation at the values of invariant functions turns pointwise choices of Lie algebra elements into complete holomorphic flows. We first give this construction for a general group action, then apply it to unipotent groups and tori.

\subsection{Entire interpolation}

\begin{lemma}\label{lem:entire}
Let $Z\subset\C^s$ be closed and discrete, and let $E$ be a finite-dimensional complex vector space. Any map $Z\to E$ extends to a holomorphic map $\C^s\to E$.
\end{lemma}
\begin{proof}
The reduced ideal sheaf $\mathcal I_Z$ is coherent, since $Z$ is locally a finite set. The exact sequence
\[
 0\longrightarrow\mathcal I_Z\longrightarrow\OO_{\C^s}
 \longrightarrow\OO_Z\longrightarrow0
\]
and Cartan's Theorem~B \cite{Forstneric} give a surjection
\[
 \OO(\C^s)\longrightarrow\OO(Z)=\prod_{z\in Z}\C.
\]
Apply this to the coordinates of $E$ in a basis. 
\end{proof}

In particular, arbitrary nonzero numbers on $Z$ can be written as exponentials of an entire function: choose a logarithm at each point and interpolate those logarithms.

\subsection{Group actions and invariant functions}
For a holomorphic action of a complex Lie group $A$ on $X$ and $\xi\in\mathfrak a=\Lie(A)$, let
\[
 V_\xi(x)=\left.\frac{d}{dt}\right|_{t=0}\exp(t\xi)x
\]
be the corresponding complete vector field. We use left actions, and compositions act from right to left. For the usual action on functions,
\[
 (g\cdot f)(x)=f(g^{-1}x),\qquad \xi\cdot f=-V_\xi f.
\]

\begin{lemma}\label{lem:groupflow}
Suppose that $Q:X\to\C^s$ is holomorphic and invariant under $A$. For a holomorphic map $F:\C^s\to\mathfrak a$, the formula
\begin{equation}\label{eq:groupflow}
 \Phi_t(x)=\exp\bigl(tF(Q(x))\bigr)x
\end{equation}
defines a complete autonomous holomorphic flow on $X$.
\end{lemma}
\begin{proof}
The exponential map and the action are holomorphic, so \eqref{eq:groupflow} is defined and holomorphic for all $(t,x)\in\C\times X$. Invariance gives $Q(\Phi_t(x))=Q(x)$. Thus
\begin{align*}
 \Phi_s(\Phi_t(x))
 &=\exp\bigl(sF(Q(x))\bigr)\exp\bigl(tF(Q(x))\bigr)x\notag\\
 &=\exp\bigl((s+t)F(Q(x))\bigr)x=\Phi_{s+t}(x).
\end{align*}
For the middle equality, the two Lie algebra elements are scalar multiples of the same element. Also $\Phi_0=\id$ and $\Phi_{-t}$ is the inverse of $\Phi_t$.

In a basis $\xi_1,\ldots,\xi_d$ of $\mathfrak a$, write $F=\sum_jF_j\xi_j$. Differentiating at zero shows that the autonomous vector field is
\begin{equation}\label{eq:variablefield}
 W=\sum_{j=1}^dF_j(Q)V_{\xi_j}.
\end{equation}
\end{proof}

The scalar shear construction is recalled in \cite[Remark~1.6]{AndristUgoliniHom}; invariant group-valued parameters also occur in \cite[Proposition~4.10]{WinkelmannRR}.

\begin{corollary}[Unipotent interpolation]\label{cor:unipotent}
Let $U$ be a connected complex unipotent algebraic group acting on $X$, and let $Q:X\to\C^s$ be a holomorphic $U$-invariant map. Suppose that $Q(x_m)$ are pairwise distinct and form a closed discrete subset of $\C^s$. Given arbitrary $u_m\in U$, there is a complete holomorphic flow whose time-one map $\Phi$ satisfies $\Phi(x_m)=u_mx_m$ for every $m$. This flow preserves $Q$ on all of $X$.
\end{corollary}
\begin{proof}
Put $\ulie=\Lie(U)$. For a complex unipotent group, exponential and logarithm give mutually inverse algebraic maps between $\ulie$ and $U$; these follow from finite matrix series in a unitriangular realization \cite{Borel}. Choose $\xi_m=\log u_m$. Lemma~\ref{lem:entire} gives $F:\C^s\to\ulie$ with $F(Q(x_m))=\xi_m$. Apply Lemma~\ref{lem:groupflow}.
\end{proof}

The dimension of $U$ does not affect the flow count: the whole Lie algebra element $\log u_m$ is interpolated at once.

If $A\simeq(\C^*)^d$ is a connected algebraic torus, its analytic exponential $\mathfrak a\to A$ is surjective. Thus the same interpolation conclusion holds for torus-valued data, provided the invariant map is constant along the whole torus action.

\subsection{Characters and weights}

For a nonzero holomorphic eigenfunction $q$ of a torus $T$, write $\chi$ for the character satisfying
\begin{equation}\label{eq:weightconvention}
 q(tx)=\chi(t)q(x).
\end{equation}
The weight $\lambda$ of $q$ for the action $(t\cdot q)(x)=q(t^{-1}x)$ is therefore $\chi^{-1}$. We use $\lambda,\mu$ in highest-weight calculations and $\chi=\lambda^{-1},\psi=\mu^{-1}$ in formulas for the action on points. In additive notation, $d\chi=-d\lambda$.

\section{Moving sequences to make function values discrete}\label{sec:compact}

To interpolate an entire function at the numbers $q(y_m)$, we need these numbers to be pairwise distinct and closed discrete. If $q$ is regular and a one-parameter unipotent subgroup acts nontrivially on it, we can obtain such points $y_m$ from a sufficiently sparse sequence $x_m$ by applying one element of a compact real form and one unipotent shear. Given nonempty Zariski-open sets $\Omega_m$, we can choose the compact group element and the shear so that $y_m\in\Omega_m$ for every $m$.

\subsection{An auxiliary embedding}
For a connected reductive group $G$ and a reductive subgroup $H$, the quotient $G/H$ is affine by Matsushima's criterion; see \cite{Arzhantsev}. The quotient is irreducible because $G$ is connected. Choose algebra generators of $\C[X]$ and include them in a finite-dimensional $G$-stable subspace $M$. The evaluation map
\[
 X\longrightarrow M^*
\]
is a closed equivariant immersion, since $\operatorname{Sym}M\to\C[X]$ is surjective. Transitivity makes the image a closed orbit. We use this embedding for the norm in Proposition~\ref{prop:position}; the module may be reducible. A positive-dimensional orbit cannot contain zero, which is fixed by $G$.

\subsection{Moving a sequence by a compact group element}

Let $G$ be connected reductive, let $K\subset G$ be a connected compact real form, and let $V$ be a finite-dimensional rational $G$-module with a $K$-invariant Hermitian norm. The following proposition gives a sufficient condition on a linear map $Q:V\to\C^s$ for its values to be distinct and divergent on a compact translate of every sufficiently sparse sequence.

\begin{proposition}\label{prop:position}
Let $X\subset V$ be a positive-dimensional closed $G$-orbit not containing zero, and let $Q:V\to\C^s$ be linear. Assume that, for every $v\ne0$, the function $k\mapsto Q(kv)$ on $K$ is not identically zero. There are strictly increasing thresholds $R_m\to\infty$, with $R_1=0$, having the following property. Given distinct $x_m\in X$ satisfying $\|x_m\|>R_m$ for $m\ge2$, and nonempty Zariski-open sets $\Omega_m\subset X$, one element $k\in K$ satisfies
\begin{equation}\label{eq:position}
 \begin{aligned}
 kx_m&\in\Omega_m,\qquad Q(kx_m)\ne0 &&(m\ge1),\\
 \|Q(kx_m)\|&>m &&(m\ge2),
 \end{aligned}
\end{equation}
and the values $Q(kx_m)$ are pairwise distinct.
\end{proposition}

The proof uses two measure estimates. The first bounds uniformly the probability that $\|Q(kv)\|$ is small for a unit vector $v$. The second shows that the set of $k$ sending a given point into a proper algebraic subset has Haar measure zero. Together they allow all the growth and open-set conditions to hold for one $k$.

\begin{lemma}\label{lem:smallball}
Let $Q:V\to\C^s$ be linear, and assume that $Q(kv)$ is not identically zero as a function of $k\in K$ for every $v\ne0$. With normalized Haar measure $\nu$ on $K$, set
\[
 \eta(\varepsilon)=\sup_{\|v\|=1}
 \nu\{k\in K:\|Q(kv)\|\le\varepsilon\}.
\]
Then $\eta(\varepsilon)\to0$ as $\varepsilon\downarrow0$.
\end{lemma}
\begin{proof}
For fixed nonzero $v$, at least one component of $Q(kv)$ is a nonzero real-analytic function on the connected real-analytic manifold $K$. The common zero set therefore has Haar measure zero. A nonzero real-analytic function has a measure-zero zero set, as follows locally from the one-variable statement and induction in real coordinate charts. Continuity of measure from above gives
\[
 \nu\{k:\|Q(kv)\|\le\varepsilon\}\longrightarrow0
\]
for fixed $v$.

Suppose the convergence were not uniform on the unit sphere. There would be $\varepsilon_j\downarrow0$, unit vectors $v_j$, and $c>0$ such that the corresponding measures are at least $c$. Pass to a subsequence with $v_j\to v$ and $\|v\|=1$. The functions $Q(kv_j)$ converge uniformly to $Q(kv)$ on $K$. If
\[
 \delta_j=\sup_{k\in K}\|Q(kv_j)-Q(kv)\|,
\]
then $\delta_j\to0$ and
\[
 \{k:\|Q(kv_j)\|\le\varepsilon_j\}
 \subseteq\{k:\|Q(kv)\|\le\varepsilon_j+\delta_j\}.
\]
The measure of the latter set tends to zero, a contradiction.
\end{proof}

\begin{lemma}\label{lem:generic}
Let $X\subset V$ be an affine homogeneous $G$-orbit, let $Y\subsetneq X$ be closed algebraic, and let $x\in X$. Then
\[
 \{k\in K:kx\in Y\}
\]
has Haar measure zero.
\end{lemma}
\begin{proof}
Choose a regular function $f$ on $X$ that vanishes on $Y$ but not identically on $X$. Such a function exists because $X$ is affine and $Y$ is proper. The regular function $g\mapsto f(gx)$ on $G$ is nonzero by transitivity. A compact real form is Zariski dense in its connected complexification \cite{Borel}, so its restriction to $K$ is not identically zero. This restriction is real analytic, and its zero set has measure zero.
\end{proof}

\begin{proof}[Proof of Proposition~\ref{prop:position}]
Choose $R_m$ recursively so that $R_m>R_{m-1}+1$ and
\[
 \eta(m/R_m)<2^{-m-2}\qquad(m\ge2).
\]
For a sequence as in the statement, linearity gives
\[
 \nu\{k:\|Q(kx_m)\|\le m\}
 \le\eta(m/R_m)<2^{-m-2}.
\]
The union of these sets has measure less than $1/4$. Outside this union the growth condition in \eqref{eq:position} holds simultaneously for all $m\ge2$.

For each $m$, the set of $k$ with $kx_m\notin\Omega_m$ has Haar measure zero by Lemma~\ref{lem:generic}. The set of $k$ with $Q(kx_1)=0$ also has measure zero, by the nonvanishing assumption and $x_1\ne0$. Finally, for $m\ne n$,
\[
 Q(kx_m)-Q(kx_n)=Q(k(x_m-x_n)).
\]
Since $x_m-x_n\ne0$, the set of $k$ with $Q(k(x_m-x_n))=0$ has Haar measure zero. There are countably many such null sets. Their removal leaves a set of positive measure from which to choose $k$. The norms $\|Q(kx_m)\|$ tend to infinity, and the values $Q(kx_m)$ are pairwise distinct. Hence they form a closed discrete subset of $\C^s$.
\end{proof}

\begin{corollary}\label{cor:irreduciblespread}
If $V$ is irreducible and $q\in V^*$ is nonzero, Proposition~\ref{prop:position} applies to the scalar map $Q=q$.
\end{corollary}
\begin{proof}
The span of the $K$-translates of $q$ is a nonzero $K$-invariant complex subspace of $V^*$. Differentiating and complexifying makes it $G$-invariant. Irreducibility gives the whole dual module. Thus no nonzero $v$ is annihilated by every translate.
\end{proof}

\subsection{Making the values of a regular function distinct and divergent}

\begin{lemma}\label{lem:invarianttuple}
Let $G$ be connected reductive, let $V$ be a finite-dimensional rational $G$-module, and let $A\subset G$ be a one-parameter unipotent subgroup. If $q_1,\ldots,q_s$ is a basis of $(V^*)^A$ and $Q=(q_1,\ldots,q_s)$, then $Q$ satisfies the hypothesis of Lemma~\ref{lem:smallball}.
\end{lemma}
\begin{proof}
Choose a direct-sum decomposition of $V^*$ into irreducible submodules. The restriction of each summand to $A\simeq(\C,+)$ has a nonzero fixed vector, since its infinitesimal action is nilpotent. The $G$-span of this vector is the entire summand. Consequently
\[
 \Span\bigl(G\cdot(V^*)^A\bigr)=V^*.
\]
The same is true with $K$ in place of $G$: the complex span of a $K$-orbit is stable under the complexified Lie algebra, and $G$ is connected. If $Q(kv)=0$ for every $k\in K$, then every $K$-translate of every $q_j$ annihilates $v$. Their span is $V^*$, so $v=0$.
\end{proof}

\begin{lemma}\label{lem:prepare}
Let $X$ be a positive-dimensional affine homogeneous space of a connected reductive group $G$. Let $q\in\C[X]$ and let $A=\exp(\C\xi)\subset G$ be a one-parameter unipotent subgroup such that $V_\xi q\not\equiv0$. Fix a closed equivariant embedding $X\subset V$ and a compact real form $K$, with a $K$-invariant Hermitian norm.

There are strictly increasing thresholds $R_m\to\infty$, with $R_1=0$, such that for every sequence of distinct points $x_m$ with $\|x_m\|>R_m$ for $m\ge2$, and every sequence of nonempty Zariski-open sets $\Omega_m\subset X$, there exist $k\in K$ and the time-one map $\Psi$ of one complete holomorphic flow for which
\[
 y_m=\Psi(kx_m)\in\Omega_m,
 \qquad |q(y_m)|>m,
\]
and the values $q(y_m)$ are pairwise distinct. The thresholds are independent of the open sets.
\end{lemma}
\begin{proof}
Use the invariant tuple $Q$ of Lemma~\ref{lem:invarianttuple}. Choose a nonzero regular function $h_m$ such that the principal open set $D(h_m)=\{h_m\ne0\}$ is nonempty and contained in $\Omega_m$. Since $X$ is irreducible and $V_\xi q$ is a nonzero regular function,
\[
 D(h_m)\cap\{V_\xi q\ne0\}
\]
is nonempty. The positive-dimensional orbit $X$ does not contain zero, so Proposition~\ref{prop:position} applies. With these open sets it gives $k\in K$ such that, writing $z_m=kx_m$, the values $Q(z_m)$ are pairwise distinct and closed discrete, and
\[
 h_m(z_m)\ne0,\qquad V_\xi q(z_m)\ne0.
\]

For each $m$ set
\[
 P_m(t)=q(\exp(t\xi)z_m),\qquad
 J_m(t)=h_m(\exp(t\xi)z_m).
\]
Both are polynomials: the $A$-action is algebraic and $q,h_m$ are regular. Moreover,
\[
 P'_m(0)=V_\xi q(z_m)\ne0,\qquad J_m(0)\ne0.
\]
Thus $P_m$ is nonconstant and $J_m$ has only finitely many zeros. Choose $t_m$ recursively so that
\[
 J_m(t_m)\ne0,\qquad |P_m(t_m)|>m,
 \qquad P_m(t_m)\notin\{P_j(t_j):j<m\}.
\]
This is possible because a nonconstant polynomial tends to infinity in modulus as $|t|\to\infty$, and each excluded finite value has finitely many preimages.

Entire interpolation gives $F:\C^s\to\C$ with $F(Q(z_m))=t_m$. The map
\[
 \Psi_\tau(x)=\exp\bigl(\tau F(Q(x))\xi\bigr)x
\]
is complete by Lemma~\ref{lem:groupflow}, because $Q$ is $A$-invariant. Its time-one map sends $z_m$ to $\exp(t_m\xi)z_m$. Since $J_m(t_m)\ne0$, these images lie in $\Omega_m$. Their $q$-values are $P_m(t_m)$, which are pairwise distinct and have modulus greater than $m$. The thresholds are those of Proposition~\ref{prop:position}.
\end{proof}

\section{Normalization with two solvable subgroups}\label{sec:solvable}

In this section $G$ is semisimple, acts effectively on $X=G/H$, and satisfies
\begin{equation}\label{eq:principalcase}
 \rk G\ge2,\qquad \dim H>0.
\end{equation}
The subgroup $H$ is closed and reductive; no connectedness assumption is made in this section. Section~\ref{sec:cases} treats the remaining cases needed for Theorem~\ref{thm:main}.

We construct connected solvable subgroups $S_1,S_2$ fixing functions $q,r$, respectively, and points $p_m$ with $q(p_m)=r(p_m)=m$. For each $m$, the set of products $a_mb_mp_m$, with $a_m\in S_1$ and $b_m\in S_2$, contains a nonempty open subset $\Omega_m$ of $X$. Lemma~\ref{lem:prepare} moves any sufficiently sparse sequence into these sets and makes its $q$-values distinct and divergent. Once $y_m=a_mb_mp_m$, we factor $a_m^{-1}$ and $b_m^{-1}$ into torus and unipotent elements and interpolate logarithms of these factors, first at $q(y_m)$, then at $r(b_mp_m)=m$.

Fix a Borel subgroup $B=TU$, where $T$ is a maximal torus and $U$ is the unipotent radical of $B$. Write $\g=\Lie(G)$, $\h=\Lie(H)$, $\tlie=\Lie(T)$, and $\ulie=\Lie(U)$. Let $\Phi$ be the root system of $(G,T)$, with positive roots $\Phi^+$ corresponding to $B$, and put $\Phi^-=-\Phi^+$ and $\ulie_-=\bigoplus_{\gamma\in\Phi^-}\g_\gamma$. Decompose $\g$ into its simple ideals,
\[
 \g=\bigoplus_{j=1}^a\g_j.
\]
As in Section~\ref{sec:flows}, $\lambda,\mu$ denote weights for the usual action on functions. We write them additively in weight identities and use $d\lambda,d\mu$ for their differentials.

\begin{lemma}\label{lem:fullsupport}
There is a nonzero regular function $q_0\in\C[X]^U$ of dominant weight $\lambda$ nonzero on each simple factor. For some negative simple-root vector $f_\alpha$, one has $V_{f_\alpha}q_0\not\equiv0$.
\end{lemma}
\begin{proof}
Each simple ideal acts nontrivially on $\C[X]$. Otherwise its connected normal subgroup would fix every regular function, hence every point of the affine variety, contrary to effectivity. The action on $\C[X]$ is locally finite. For each $\g_j$, choose a finite-dimensional irreducible submodule on which $\g_j$ acts nontrivially, and take a nonzero highest-weight function $q_j$ in it. Its dominant highest weight $\lambda_j$ has nonzero component on $\g_j$.

The product
\[
 q_0=\prod_{j=1}^a q_j
\]
is nonzero because $\C[X]$ is a domain. It is $U$-fixed and has dominant weight $\lambda=\sum_j\lambda_j$. Dominant components cannot cancel, so $\lambda$ is nonzero on every factor. In particular $q_0$ is nonconstant. It lies in a finite-dimensional rational submodule by local finiteness.

Choose a simple root $\alpha$ with $c=\langle\lambda,\alpha^\vee\rangle>0$. For an $\mathfrak{sl}_2$-triple $(e_\alpha,f_\alpha,h_\alpha)$, the usual action satisfies
\[
 e_\alpha\cdot(f_\alpha\cdot q_0)
 =[e_\alpha,f_\alpha]\cdot q_0
 =c q_0\ne0,
\]
since $e_\alpha\cdot q_0=0$. Therefore $f_\alpha\cdot q_0\ne0$, equivalently $V_{f_\alpha}q_0\ne0$.
\end{proof}

\begin{lemma}\label{lem:transverse}
With $q_0$ and $\lambda$ as above, there exist a Weyl element $w$, a representative $n_w\in N_G(T)$, and connected solvable subgroups
\[
 S_1=U(\ker\lambda)^\circ,
 \qquad S_2=n_wS_1n_w^{-1},
\]
with the following properties. The weights $\lambda$ and $\mu=w\lambda$ are linearly independent. The function $r_0=n_w\cdot q_0$ is fixed by $S_2$ and has weight $\mu$. The locus
\begin{equation}\label{eq:transverselocus}
 \mathcal U=\{x\in X:\salg_1+\salg_2+\g_x=\g\}
\end{equation}
is a nonempty Zariski-open subset of $X$, where $\salg_i=\Lie(S_i)$ and $\g_x$ is the stabilizer Lie algebra at $x$.
\end{lemma}
\begin{proof}
The subgroup $S_1$ fixes $q_0$ because $U$ does and $\lambda$ is trivial on $(\ker\lambda)^\circ$. Conjugation then shows that $S_2$ fixes $r_0$. Both subgroups are normalized by $T$. Let $w_0$ be the longest Weyl element.

We begin with $w_0\lambda=-\lambda$. For an irreducible root system, $w_0=-1$ in types $A_1$, $B_\ell$, $C_\ell$, $D_{2\ell}$, $E_7$, $E_8$, $F_4$, and $G_2$; see \cite{Humphreys}. Thus this case includes every weight in these types, as well as the self-dual highest weights in the remaining types. We will choose $w=w_0s_\alpha$ so that $\lambda,w\lambda$ are independent and the sum of the two solvable Lie algebras omits just one root space. A conjugate of $\h$ will have a nonzero component in that space.

Choose a simple ideal $\g_j$ on which $\h$ has nonzero projection. Such an ideal exists because $\dim H>0$. Since $\lambda$ is nonzero on this factor, there is a simple root $\alpha$ in it with
\[
 c=\langle\lambda,\alpha^\vee\rangle>0.
\]
Set
\begin{equation}\label{eq:weylchoice}
 w=w_0s_\alpha,\qquad \beta=-w_0\alpha.
\end{equation}
Then $\beta$ is simple positive, and
\begin{equation}\label{eq:weightshift}
 \mu=w\lambda=-\lambda+c\beta.
\end{equation}
The weight $\lambda$ is not proportional to $\beta$. If there is more than one simple factor, $\lambda$ is nonzero on every factor, while $\beta$ is supported in just one. If there is one factor, it has rank at least two. A positive multiple of its simple root $\beta$ pairs negatively with the coroot of a neighboring simple root, while a negative multiple pairs negatively with $\beta^\vee$. Neither is dominant. Equation~\eqref{eq:weightshift} therefore makes $\lambda$ and $\mu$ independent.

A simple reflection preserves all positive roots except its own, which it sends to its negative. Consequently
\[
 w(\Phi^+)=\bigl(\Phi^-\setminus\{-\beta\}\bigr)\cup\{\beta\}.
\]
The Lie algebras of the two solvable subgroups are
\[
 \salg_1=\ulie\oplus\ker d\lambda,\qquad
 \salg_2=\Ad(n_w)\ulie\oplus\ker d\mu.
\]
Since $\lambda$ and $\mu$ are independent, their differentials have distinct kernels, and
\[
 \ker d\lambda+\ker d\mu=\tlie.
\]
Combining this equality with the root-set identity above gives
\begin{equation}\label{eq:hyperplane}
 \mathfrak m:=\salg_1+\salg_2
 =\tlie\oplus\bigoplus_{\gamma\in\Phi,\ \gamma\ne-\beta}\g_\gamma.
\end{equation}
The vector subspace $\mathfrak m$ has codimension one.

Choose $Z\in\h$ with nonzero projection to $\g_j$. The span of $\Ad(G)Z$ is an ideal containing $\g_j$. Indeed, brackets with $\g_j$ yield a nonzero subspace of $\g_j$, and simplicity then gives the entire ideal. Thus the $(-\beta)$-root coefficient cannot vanish on every $\Ad(g)Z$. For some $g\in G$, the stabilizer algebra $\Ad(g)\h$ has a vector outside $\mathfrak m$, and
\[
 \mathfrak m+\g_{gH}=\g.
\]
It remains to consider $w_0\lambda\ne-\lambda$. In this case $\lambda$ and $w_0\lambda$ are independent: dependence and $w_0^2=1$ would imply $w_0\lambda=\pm\lambda$, while $w_0\lambda=\lambda$ would make the nonzero weight $\lambda$ both dominant and antidominant. Take $w=w_0$. Then $\salg_1=\ulie\oplus\ker d\lambda$ and $\salg_2=\ulie_-\oplus\ker d(w_0\lambda)$. The two Cartan kernels span $\tlie$, and the radicals contain all root spaces, so $\mathfrak m:=\salg_1+\salg_2=\g$ and $\mathcal U=X$.

Finally, $\mathcal U$ is the surjectivity locus of the algebraic bundle map $\mathfrak m\otimes\mathcal O_X\to TX$ that evaluates action fields. Surjectivity is an open rank condition. In the first case the point $gH$ constructed above belongs to $\mathcal U$; in the second case $\mathcal U=X$. Thus the open set is nonempty in both cases.
\end{proof}

For example, on $\SL_3/\{\diag(t,t^{-1},1)\}$ the construction gives
\[
 \salg_1+\salg_2=\{Z\in\mathfrak{sl}_3:Z_{32}=0\}.
\]
A conjugate of $\Lie(H)$ contains a vector with nonzero $(3,2)$ entry, so its sum with $\salg_1+\salg_2$ is $\mathfrak{sl}_3$. Section~\ref{sec:cstar} gives the matrices and the resulting flow maps.

Put $\chi=\lambda^{-1}$ and $\psi=\mu^{-1}$. These are the characters of $q_0,r_0$ in \eqref{eq:weightconvention}, and their differentials are independent.

\begin{proposition}\label{prop:model}
There are regular eigenfunctions $q,r$, a point $o\in X$, and a sequence $p_m\in T\cdot o$ such that $q(p_m)=r(p_m)=m$ and every map
\begin{equation}\label{eq:productmap}
 \vartheta_m:S_1\times S_2\longrightarrow X,
 \qquad (a,b)\longmapsto abp_m
\end{equation}
is dominant. Its image contains a nonempty Zariski-open subset $\Omega_m\subset X$. The sequence $\{p_m\}$ is closed discrete.
\end{proposition}
\begin{proof}
Choose
\[
 o\in\mathcal U\cap\{q_0r_0\ne0\}.
\]
This intersection is nonempty because $q_0r_0$ is nonzero and $X$ is irreducible. Set $q=q_0/q_0(o)$ and $r=r_0/r_0(o)$. Rescaling preserves the weights and the subgroups fixing the functions.

Since $d\chi,d\psi$ are independent, choose $C\in\tlie$ with
\[
 d\chi(C)=d\psi(C)=1,
 \qquad t_m=\exp((\log m)C),\qquad p_m=t_mo.
\]
On points,
\[
 q(tx)=\chi(t)q(x),\qquad r(tx)=\psi(t)r(x),
\]
so $q(p_m)=r(p_m)=m$. These equalities show that the points are distinct and have no accumulation point in $X$.

Because $\mathfrak m=\salg_1+\salg_2$ is $T$-stable and $\g_{p_m}=\Ad(t_m)\g_o$,
\[
 \mathfrak m+\g_{p_m}=\Ad(t_m)(\mathfrak m+\g_o)=\g.
\]
The differential of \eqref{eq:productmap} at $(1,1)$ is
\[
 (\xi,\eta)\longmapsto V_\xi(p_m)+V_\eta(p_m).
\]
Its image is $(\mathfrak m+\g_{p_m})/\g_{p_m}=T_{p_m}X$. Thus the map is dominant. Chevalley's theorem makes its image constructible, and a dense constructible subset of an irreducible variety contains a nonempty open subset \cite{Stacks}. Choose one such $\Omega_m$ for each $m$.
\end{proof}

\subsection{Realizing the factors by complete flows}

\begin{lemma}\label{lem:solvable}
Let $S=UT_0$ be a connected solvable algebraic subgroup, where $U$ is unipotent, $T_0$ is a connected torus, and $T_0$ normalizes $U$. Suppose that $S$ fixes a holomorphic function $f$ on $X$. Let $x_m\in X$ have pairwise distinct values $f(x_m)$ forming a closed discrete subset of $\C$. For arbitrary $s_m\in S$, there exist complete holomorphic flow maps $\Phi_1,\Phi_2$ such that
\[
 (\Phi_2\circ\Phi_1)(x_m)=s_mx_m\qquad(m\ge1).
\]
Each flow preserves $f$ on all of $X$.
\end{lemma}
\begin{proof}
Write $s_m=u_ma_m$ with $u_m\in U$ and $a_m\in T_0$. Choose logarithms $A_m\in\Lie(T_0)$ with $\exp(A_m)=a_m$, and $B_m\in\Lie(U)$ with $\exp(B_m)=u_m$. Torus exponentials are analytically surjective; for a unipotent group exponential and logarithm are polynomial inverses. Choose entire maps $A:\C\to\Lie(T_0)$ and $B:\C\to\Lie(U)$ satisfying $A(f(x_m))=A_m$ and $B(f(x_m))=B_m$.

Apply first the flow $x\mapsto\exp(tA(f(x)))x$. At time one it sends $x_m$ to $a_mx_m$ and preserves $f$. Apply next $x\mapsto\exp(tB(f(x)))x$. Since $f(a_mx_m)=f(x_m)$, its time-one map sends $a_mx_m$ to $u_ma_mx_m$. Both flows are complete by Lemma~\ref{lem:groupflow}.
\end{proof}

\subsection{Interpolation on a torus orbit}

\begin{lemma}\label{lem:torusproduct}
Let $\chi,\psi:T\to\C^*$ be characters whose differentials are linearly independent in $\tlie^*$. Put
\[
 T_q=(\ker\chi)^\circ,\qquad T_r=(\ker\psi)^\circ.
\]
Then $T=T_rT_q$.
\end{lemma}
\begin{proof}
The Lie algebras are $\mathfrak t_q=\ker d\chi$ and $\mathfrak t_r=\ker d\psi$. Independent linear forms have distinct kernels, whose sum is $\tlie$. Given $t\in T$, choose a logarithm $L\in\tlie$ and decompose $L=L_r+L_q$ with $L_r\in\mathfrak t_r$ and $L_q\in\mathfrak t_q$. Since $T$ is abelian,
\[
 t=\exp(L_r)\exp(L_q),\qquad
 \exp(L_r)\in T_r,\quad\exp(L_q)\in T_q.
\]
Taking identity components ensures that the factors have logarithms in their Lie algebras, even when the characters are not primitive.
\end{proof}

\begin{proposition}[Interpolation in one torus orbit]\label{prop:torusmatch}
Let $T$ act holomorphically on $X$, and let $q,r\in\OO(X)$ satisfy
\[
 q(tx)=\chi(t)q(x),\qquad r(tx)=\psi(t)r(x),
\]
where $d\chi,d\psi$ are independent. Let $a_m$ and $b_m$ lie in one $T$-orbit. Suppose that each of the two sequences $q(a_m)$ and $r(b_m)$ consists of pairwise distinct numbers forming a closed discrete subset of $\C$. Then there exist complete holomorphic flow maps $\Phi_1,\Phi_2$ such that
\[
 (\Phi_2\circ\Phi_1)(a_m)=b_m\qquad(m\ge1).
\]
\end{proposition}
\begin{proof}
Choose $\tau_m\in T$ with $\tau_ma_m=b_m$. Factor it, by Lemma~\ref{lem:torusproduct}, as
\[
 \tau_m=v_mw_m,\qquad v_m\in T_r,\quad w_m\in T_q.
\]
Choose logarithms $A_m\in\mathfrak t_q$ of $w_m$. Interpolate a holomorphic map $A:\C\to\mathfrak t_q$ with $A(q(a_m))=A_m$. Because $T_q$ preserves $q$, Lemma~\ref{lem:groupflow} gives a complete flow with time-one map
\[
 a_m\longmapsto a'_m=w_ma_m.
\]
Since $v_m\in T_r$, we have
\begin{equation}\label{eq:intermediatelabel}
 r(a'_m)=r(v_ma'_m)=r(b_m).
\end{equation}

Choose logarithms $B_m\in\mathfrak t_r$ of $v_m$ and an entire map $B:\C\to\mathfrak t_r$ satisfying $B(r(b_m))=B_m$. Since $T_r$ preserves $r$, Lemma~\ref{lem:groupflow} gives the complete flow $x\mapsto\exp(tB(r(x)))x$. Its time-one map sends $a'_m$ to $v_ma'_m=b_m$.
\end{proof}

\subsection{Normalizing a sparse sequence}

\begin{proposition}\label{prop:principal}
Under \eqref{eq:principalcase}, there are a closed equivariant embedding $X\subset V$, a compact real form $K$, a $K$-invariant Hermitian norm, and increasing thresholds $R_m\to\infty$, with $R_1=0$, such that every sequence of distinct points $x_m$ with $\|x_m\|>R_m$ for $m\ge2$ is sent to the sequence of Proposition~\ref{prop:model} by the action of one element of $K$ followed by at most five complete holomorphic flow maps. Every injective self-map of this sequence is realized by two complete flow maps.
\end{proposition}
\begin{proof}
Apply Lemma~\ref{lem:prepare} to the function $q$ and the open sets $\Omega_m$ of Proposition~\ref{prop:model}. Lemma~\ref{lem:fullsupport} provides the negative simple-root subgroup moving $q$ nontrivially. We obtain the action of one element of $K$ and one unipotent shear sending $x_m$ to $y_m\in\Omega_m$, with pairwise distinct divergent values $q(y_m)$.

Choose independently for each $m$ a factorization
\begin{equation}\label{eq:chosenfactors}
 y_m=a_mb_mp_m,\qquad a_m\in S_1,\quad b_m\in S_2.
\end{equation}
The connected subgroup $S_1=U(\ker\chi)^\circ$ fixes $q$. Apply Lemma~\ref{lem:solvable} to $S_1$, the function $q$, and the elements $a_m^{-1}$. Two flow maps send
\[
 y_m\longmapsto z_m=b_mp_m.
\]
The whole subgroup $S_2$ fixes $r$, so
\begin{equation}\label{eq:targetlabels}
 r(z_m)=r(b_mp_m)=r(p_m)=m.
\end{equation}
The values $r(z_m)$ are therefore pairwise distinct and closed discrete. Two further flow maps from Lemma~\ref{lem:solvable}, now for $S_2$ with elements $b_m^{-1}$, send $z_m$ to $p_m$. In total, there are five flows after the compact action.

For a self-injection $p_m\mapsto p_{\sigma(m)}$, use Proposition~\ref{prop:torusmatch} with the torus $T$ and the independent characters $\chi,\psi$. The two sequences of numbers $q(p_m)=m$ and $r(p_{\sigma(m)})=\sigma(m)$ are each pairwise distinct and closed discrete. The proposition gives the two flows.
\end{proof}

\needspace{9\baselineskip}
\section{A nonspherical example: \texorpdfstring{$\SL_3/\C^*$}{SL3/C*}}\label{sec:cstar}

Let
\[
 G=\SL_3(\C),\qquad H=\{\diag(t,t^{-1},1):t\in\C^*\},
 \qquad X=G/H.
\]
The quotient has dimension seven, while a Borel subgroup of $G$ has dimension five, so no Borel orbit is open. We give explicit formulas for the maps in Proposition~\ref{prop:principal}. The calculation also shows how a conjugate of $\Lie(H)$ completes the sum of the two solvable Lie algebras to $\g$.

\subsection{Affine coordinates and the two functions}
Write $g=[v_1\ v_2\ c]\in\SL_3$ in columns. Right multiplication by $H$ scales $v_1$ and $v_2$ inversely, so
\[
 M=v_1v_2^{\mathsf T},\qquad c
\]
are coordinates on the quotient. For a matrix $M$, put
\[
 \ell(M)=(M_{23}-M_{32},\ M_{31}-M_{13},\ M_{12}-M_{21}).
\]
The row $\ell(M)$ represents $v_1\times v_2$, and $\ell(M)c=\det g=1$. We obtain the closed equivariant realization
\begin{equation}\label{eq:cstarembedding}
 X\simeq\{(M,c)\in\operatorname{Mat}_3(\C)\oplus\C^3:
                  \rk M\le1,\ \ell(M)c=1\},
 \qquad h\cdot(M,c)=(hMh^{\mathsf T},hc).
\end{equation}
Indeed, the equation $\ell(M)c=1$ forces $M\ne0$. A rank-one factorization $M=v_1v_2^{\mathsf T}$ then gives a matrix $[v_1\ v_2\ c]$ of determinant one, unique up to $H$. These factorizations are regular on the patches where an entry of $M$ is nonzero, and their cosets agree on overlaps. This proves the asserted isomorphism.

Set $P=c\ell(M)$. Under $G$, it transforms by conjugation. Define
\begin{equation}\label{eq:cstarqr}
 q=-P_{31}=-c_3\ell_1,\qquad r=P_{12}=c_1\ell_2.
\end{equation}
Let $U_1$ be the upper unitriangular subgroup and set
\[
 n=\begin{pmatrix}0&0&1\\1&0&0\\0&1&0\end{pmatrix},\qquad
 U_2=nU_1n^{-1}
 =\left\{\begin{pmatrix}1&0&0\\v&1&u\\w&0&1\end{pmatrix}:u,v,w\in\C\right\}.
\]
The function $q$ is fixed by $U_1$, and $r=-n\cdot q$ by $U_2$. For $t=\diag(t_1,t_2,t_3)$, their characters in \eqref{eq:weightconvention} are $t_3/t_1$ and $t_1/t_2$. Thus $q$ and $r$ are fixed, respectively, by the connected solvable groups
\[
 S_1=U_1T_q,\quad T_q=\{\diag(a,a^{-2},a):a\in\C^*\},
 \qquad
 S_2=U_2T_r,\quad T_r=\{\diag(b,b,b^{-2}):b\in\C^*\}.
\]

\subsection{The omitted root space and the fixed sequence}
The two Lie algebras satisfy
\[
 \salg_1+\salg_2=\{Z\in\mathfrak{sl}_3:Z_{32}=0\}.
\]
Thus the sum omits the root space $\C E_{32}$. Choose
\[
 g_0=\begin{pmatrix}1&1&1\\0&1&1\\1&0&1\end{pmatrix},\qquad o=g_0H.
\]
At $o$ we have $c=(1,1,1)^{\mathsf T}$ and $\ell=(-1,1,1)$, so $q(o)=r(o)=1$. Moreover,
\[
 \Ad(g_0)\diag(1,-1,0)
 =\begin{pmatrix}0&-1&1\\-1&0&1\\1&-1&0\end{pmatrix}\in\g_o.
\]
Its $(3,2)$ entry is $-1$, and consequently $\salg_1+\salg_2+\g_o=\g$. In the notation of Lemma~\ref{lem:transverse}, $\lambda=\alpha_1+\alpha_2$, the Weyl element is $w_0s_{\alpha_1}$, and $\beta=\alpha_2$.

Put
\begin{equation}\label{eq:cstarmodel}
 t_m=\diag(1,m^{-1},m),\qquad p_m=t_mg_0H.
\end{equation}
Then $q(p_m)=r(p_m)=m$, so the points $p_m$ form a closed discrete set. Since the torus normalizes $S_1$ and $S_2$, the equality $\salg_1+\salg_2+\g_{p_m}=\g$ holds for every $m$.

\subsection{A unipotent shear and an explicit factorization}
To apply Lemma~\ref{lem:prepare}, take $A=\exp(\C E_{21})$. In the embedding \eqref{eq:cstarembedding}, a basis of the invariant linear functions gives
\begin{equation}\label{eq:cstarQ}
 Q=(M_{11},M_{13},M_{31},M_{33},M_{12}-M_{21},c_1,c_3).
\end{equation}
The action adds $t$ times row one to row two of $M$, and then the same multiple of column one to column two. This verifies the invariance of $Q$. Also,
\[
 q(\exp(tE_{21})x)=q(x)+t c_3\ell_2,
\]
whose coefficient is one at $o$. Lemma~\ref{lem:prepare} therefore applies with this $Q$ and $q$.

In this example we can specify one open set that works for every $m$. Let $\Omega$ consist of the cosets represented by $g=(g_{ij})$ with
\begin{equation}\label{eq:cstaropen}
 g_{31}g_{33}(g_{23}g_{31}-g_{21}g_{33})
                (g_{33}^2-g_{31}g_{32})\ne0.
\end{equation}
The nonvanishing conditions are unchanged under right multiplication by $H$. They define a nonempty Zariski-open subset of $X$, since they hold at $g_0$.

Fix $gH\in\Omega$ and $m\ge1$. We may replace $g$ by $gh(s)^{-1}$, where $h(s)=\diag(s,s^{-1},1)$, without changing the coset. We choose $s$ so that the matrix to be factored has zero $(3,2)$ entry. For $s\ne0$, put
\[
 C(s)=g h(s)^{-1}g_0^{-1}t_m^{-1}.
\]
Direct multiplication gives
\[
 C(s)_{32}=m\left(g_{33}-\frac{g_{31}}s\right).
\]
Thus $s=g_{31}/g_{33}$ makes this entry zero. Set $C=C(s)$ for this choice of $s$. The matrix $C$ is independent of the original representative $g$, and $gH=Cp_m$. Its diagonal entries satisfy
\[
 C_{22}=m\left(g_{23}-\frac{g_{21}g_{33}}{g_{31}}\right),\qquad
 C_{33}=\frac1m\left(g_{33}-\frac{g_{31}g_{32}}{g_{33}}\right).
\]
Both diagonal entries are nonzero by \eqref{eq:cstaropen}. Choose a cube root and define
\begin{align*}
 a^3&=(C_{22}^2C_{33})^{-1},& b&=C_{22}a^2,\\
 \xi&=C_{12}/C_{22},& \zeta&=(C_{13}-\xi C_{23})/C_{33},\\
 v&=C_{21}/C_{22},& w&=C_{31}/(ab),&
 u&=C_{23}a^2b^2.
\end{align*}
Then
\begin{equation}\label{eq:cstarfactor}
 C=U^{(1)}D_aU^{(2)}D_b,
\end{equation}
where
\begin{align*}
 U^{(1)}&=I+\xi E_{12}+\zeta E_{13},& D_a&=\diag(a,a^{-2},a),\\
 U^{(2)}&=I+vE_{21}+uE_{23}+wE_{31},& D_b&=\diag(b,b,b^{-2}).
\end{align*}
Multiplication verifies every entry except $C_{11}$ directly; that entry follows from $\det C=1$ and $C_{22}C_{33}\ne0$. Thus $\Omega\subset S_1S_2p_m$ for every $m$. 

\subsection{The complete flows}
Use an $\SU_3$-invariant norm on the module in \eqref{eq:cstarembedding}. For a sequence beyond the thresholds of Lemma~\ref{lem:prepare}, the first two flows act by an element of $\SU_3$ and by the unipotent shear
\[
 x\longmapsto\exp\bigl(F(Q(x))E_{21}\bigr)x.
\]
Here $F$ is entire and is chosen by that lemma so that the resulting points $y_m$ lie in $\Omega$ and the values $q(y_m)$ are pairwise distinct and divergent. Apply \eqref{eq:cstarfactor} at each $y_m$ to write
\[
 y_m=U^{(1)}_mD_{a_m}U^{(2)}_mD_{b_m}p_m.
\]
Set
\[
 H_q=\diag(1,-2,1),\qquad H_r=\diag(1,1,-2),
\]
and note that
\[
 \log U^{(1)}_m=\xi_mE_{12}+\zeta_mE_{13},\qquad
 \log U^{(2)}_m=(v_m-u_mw_m/2)E_{21}+u_mE_{23}+w_mE_{31}.
\]
For $j=3,4,5,6$, the row labelled $j$ in the following table prescribes the values of an entire map $L_j$ into the Lie algebra of the group listed in that row. The corresponding flow is $x\mapsto\exp(\tau L_j(f(x)))x$, where $f=q$ for $j=3,4$ and $f=r$ for $j=5,6$; in each row the acting subgroup fixes $f$.
\begin{center}
\small
Flows sending $y_m$ to $p_m$\par\smallskip
\begin{tabular}{@{}cllll@{}}
\toprule
Flow $j$ & Group & Argument $z$ & $L_j(z)$ & Image of the $m$th point\\
\midrule
3 & $U_1$ & $q(y_m)$ & $-\log U^{(1)}_m$ & $D_{a_m}U^{(2)}_mD_{b_m}p_m$\\
4 & $T_q$ & $q(y_m)$ & $-(\log a_m)H_q$ & $U^{(2)}_mD_{b_m}p_m$\\
5 & $U_2$ & $m$ & $-\log U^{(2)}_m$ & $D_{b_m}p_m$\\
6 & $T_r$ & $m$ & $-(\log b_m)H_r$ & $p_m$\\
\bottomrule
\end{tabular}
\end{center}
Flows 3 and 4 preserve $q$ and send $y_m$ to $U^{(2)}_mD_{b_m}p_m$. Since $U^{(2)}_mD_{b_m}\in S_2$, the value of $r$ at this point is $m$. Flows 5 and 6 preserve $r$. Each $L_j$ exists by entire interpolation, and Lemma~\ref{lem:groupflow} proves completeness. The divergence calculation in Section~\ref{sec:tameness} proves that these flows preserve the invariant volume form. Finally, Proposition~\ref{prop:torusmatch}, applied to \eqref{eq:cstarmodel}, realizes every injection $p_m\mapsto p_{\sigma(m)}$ by two torus flows.

\section{The remaining homogeneous spaces}\label{sec:cases}

We reduce an arbitrary $X$ in Theorem~\ref{thm:main} to Proposition~\ref{prop:principal} or to the rank-one quadric.

\subsection{Removing the action kernel}
Let
\[
 N=\bigcap_{g\in G}gHg^{-1}
\]
be the kernel of the $G$-action on $X$. It is a closed normal algebraic subgroup contained in $H$. Put $\bar G=G/N$ and $\bar H=H/N$. Then $\bar G$ is connected semisimple, $\bar H$ is connected reductive, and
\[
 X\simeq\bar G/\bar H.
\]
Connectedness of $\bar H$ follows from that of $H$, and reductivity follows because an algebraic image of a reductive group in characteristic zero is reductive; see \cite{Borel}.

If $\rk\bar G\ge2$ and $\dim\bar H>0$, Proposition~\ref{prop:principal} applies. We address the other two possibilities directly.

\subsection{The group case through a left-right action}
If $\bar H$ has dimension zero, connectedness makes it trivial. Thus $X=\bar G=:J$ is a positive-dimensional semisimple algebraic group. Let $J\times J$ act by
\[
 (a,b)\cdot x=axb^{-1}.
\]
The stabilizer of the identity is $\Delta J$, and the action kernel is $\Delta Z(J)$. The effective acting group
\[
 (J\times J)/\Delta Z(J)
\]
is connected semisimple of rank $2\rk J\ge2$, and its stabilizer is the connected reductive positive-dimensional group $\Delta J/\Delta Z(J)$. Proposition~\ref{prop:principal} therefore applies to this action and gives the normalizing maps and the extensions of self-injections with the same bounds. The RR assertion for semisimple groups was already known \cite[Theorems~2.3(ii) and~3.1]{WinkelmannRR}; the left-right presentation gives the stated flow bound.

\subsection{The effective rank-one quotient}\label{subsec:quadric}
Suppose $\rk\bar G=1$ and $\dim\bar H>0$. A proper connected reductive subgroup of positive dimension in a complex rank-one semisimple group is a maximal torus: its Lie algebra is either a one-dimensional torus or all of $\mathfrak{sl}_2$, and the latter would give a point quotient. Pulling back to the simply connected cover identifies
\[
 X\simeq\SL_2/T.
\]
The inverse image of the maximal torus is a maximal torus containing the covering kernel, so the quotients are isomorphic. Conjugating that torus to $\SO_2(\C)$ gives the symmetric-matrix realization
\begin{equation}\label{eq:quadric}
 X=\left\{\begin{pmatrix}a&b\\b&c\end{pmatrix}:ac-b^2=1\right\},
 \qquad g\cdot S=gSg^{\mathsf T}.
\end{equation}
The RR property of this quadric follows from Winkelmann's results on Danielewski surfaces \cite[Theorem~2.3(iii)]{WinkelmannRR}. We give explicit normalizing flows and extensions of self-injections.

With $E_{ij}$ denoting the elementary matrix, write $u_+(t)=I+tE_{12}$ and $u_-(s)=I+sE_{21}$. The actions on coordinates are
\begin{align}
 u_+(t):\ (a,b,c)&\longmapsto(a+2tb+t^2c,\ b+tc,\ c),\label{eq:upper2}\\
 u_-(s):\ (a,b,c)&\longmapsto(a,\ b+sa,\ c+2sb+s^2a).\label{eq:lower2}
\end{align}
Thus the upper action preserves $c$, and the lower action preserves $a$.

Choose $\beta_m\in\C$ once and for all with $\beta_m^2=m^2-1$, and put
\begin{equation}\label{eq:quadricmodel}
 p_m=\begin{pmatrix}m&\beta_m\\\beta_m&m\end{pmatrix}.
\end{equation}
The diagonal entries show that these points form a closed discrete sequence.

The space of symmetric $2\times2$ matrices is an irreducible $\SL_2$-module. Apply Proposition~\ref{prop:position} and Corollary~\ref{cor:irreduciblespread} with the nonzero linear function $q(S)=c$ and the open set $\{c\ne0\}$. After one compact congruence, every sufficiently sparse sequence of distinct points has $c_m\ne0$ with the $c_m$ pairwise distinct and closed discrete.

For each $S_m=(a_m,b_m,c_m)$ choose $d_m\in\C$ with
\[
 d_m^2=mc_m-1,
 \qquad t_m=\frac{d_m-b_m}{c_m}.
\]
Interpolate $F(c_m)=t_m$ by an entire function and use the complete upper shear with coefficient $F(c)$. Formula~\eqref{eq:upper2} gives new coordinates $b'_m=d_m$ and $c'_m=c_m$. Since the determinant is one,
\[
 a'_m=\frac{1+d_m^2}{c_m}=m.
\]
Now interpolate
\[
 F_-(m)=\frac{\beta_m-d_m}{m}
\]
and use the complete lower shear with coefficient $F_-(a)$. It leaves $a=m$, sends $b$ to $\beta_m$, and forces
\[
 c=\frac{1+\beta_m^2}{m}=m.
\]
Thus two unipotent shears after the compact action normalize the sequence to \eqref{eq:quadricmodel}.

For an injection $\sigma:\N\to\N$, begin at $p_m$ and choose $d_m^2=m\sigma(m)-1$. An upper shear with time $(d_m-\beta_m)/m$, interpolated at $c=m$, gives $(a,b,c)=(\sigma(m),d_m,m)$. A lower shear with time
\[
 \frac{\beta_{\sigma(m)}-d_m}{\sigma(m)},
\]
interpolated at $a=\sigma(m)$, gives $p_{\sigma(m)}$. The numbers $\sigma(m)$ are pairwise distinct and form a closed discrete subset of $\C$.

In all cases, sufficiently sparse sequences can be sent to the fixed sequence, and its injective self-maps extend by two flows. Proposition~\ref{prop:principal} gives at most five flows after one compact action, and the quadric uses the upper and lower shears. An element of a connected compact Lie group belongs to a maximal torus and is an exponential. Its action is therefore itself one complete flow map, giving $n\le6$ in Theorem~\ref{thm:main}\textup{(a)}.

\subsection{The antipodal quotient}\label{subsec:antipodal}
The other positive-dimensional reductive isotropy in rank one is $N(T)$. Indeed, the identity component of a proper positive-dimensional reductive subgroup is a maximal torus, and the subgroup lies in its normalizer. Since $N(T)/T$ has order two, the possibilities are $T$ and $N(T)$; compare \cite[Lemma~2.1]{AndristUgoliniHom}.

Let $\widetilde X$ be the quadric in \eqref{eq:quadric}. After conjugating $T$ to $\SO_2$, the matrix $j=\diag(i,-i)$ represents the nontrivial element of $N(T)/T$ and satisfies $jj^{\mathsf T}=-I$. Right multiplication by $j$ therefore acts on $S=gg^{\mathsf T}$ by $\tau(S)=-S$. Hence
\begin{equation}\label{eq:conicquotient}
 X_N:=\SL_2/N(T)\simeq\widetilde X/\langle\tau\rangle
 \simeq\mathbf P^2\setminus \mathcal Q,
 \qquad \mathcal Q=\{ac-b^2=0\}.
\end{equation}
The last map is projectivization: a projective form of nonzero determinant has two representatives of determinant one, differing by sign. Andrist--Ugolini use the same quotient in constructing strongly tame sets \cite[Remark~2.6]{AndristUgoliniHom}. Here we obtain sparse normalization and, through Section~\ref{sec:tameness}, the full RR property.

Use the exhaustion $\rho([S])=\|S\|$ for an $\SU_2$-invariant Hermitian norm. It is well-defined and proper because the norm is sign-invariant and the quotient is finite. For distinct quotient points choose lifts $S_m$ arbitrarily. Both $S_m-S_n$ and $S_m+S_n$ are nonzero when $m\ne n$. In the proof of Proposition~\ref{prop:position}, exclude the zero sets of
\[
 k\longmapsto c\bigl(k(S_m-S_n)k^{\mathsf T}\bigr),\qquad
 k\longmapsto c\bigl(k(S_m+S_n)k^{\mathsf T}\bigr).
\]
Each has Haar measure zero by irreducibility of $\Sym_2(\C)$. The same small-ball estimate gives thresholds such that one compact congruence makes $c_m\ne0$ for all $m$, $|c_m|>m$ for $m\ge2$, and the $c_m^2$ pairwise distinct. Thus $c_m^2$ are closed discrete.

For entire functions $F_+,F_-$ the flows
\begin{align}
 S&\longmapsto u_+\bigl(tF_+(c^2)\bigr)S
                  u_+\bigl(tF_+(c^2)\bigr)^{\mathsf T},\label{eq:even-upper}\\
 S&\longmapsto u_-\bigl(tF_-(a^2)\bigr)S
                  u_-\bigl(tF_-(a^2)\bigr)^{\mathsf T}\label{eq:even-lower}
\end{align}
are complete by Lemma~\ref{lem:groupflow}. Their coefficients are unchanged by $S\mapsto-S$, so both flows commute with $\tau$ and descend to $X_N$.

Write $P_m$ for the matrix in \eqref{eq:quadricmodel} and set $\bar p_m=[P_m]$. Choose $d_m^2=mc_m-1$ and interpolate
\[
 F_+(c_m^2)=\frac{d_m-b_m}{c_m},\qquad
 F_-(m^2)=\frac{\beta_m-d_m}{m}.
\]
The calculation in Section~\ref{subsec:quadric} sends the selected lift $S_m$ first to $(m,d_m,c_m)$ and then to $P_m$. Equivariance sends the other lift to $-P_m$. Thus one compact flow and two descended shears normalize the quotient sequence to $\bar p_m$. The function $a^2$ descends to $X_N$ and takes the values $m^2$ at $\bar p_m$, so $\{\bar p_m:m\in\N\}$ is closed discrete.

For an injection $\bar p_m\mapsto\bar p_{\sigma(m)}$, choose $d_m^2=m\sigma(m)-1$ and interpolate
\[
 F_+(m^2)=\frac{d_m-\beta_m}{m},\qquad
 F_-(\sigma(m)^2)=\frac{\beta_{\sigma(m)}-d_m}{\sigma(m)}.
\]
The two descended shears send $\bar p_m$ to $\bar p_{\sigma(m)}$. The sets $\{m^2\}$ and $\{\sigma(m)^2\}$ are closed discrete, and injectivity of $\sigma$ prevents repeated arguments in the second interpolation.

\section{Volume preservation and tameness}\label{sec:tameness}

For the spaces in Theorem~\ref{thm:main}, use the closed embedding and norm chosen in Sections~\ref{sec:solvable} and~\ref{sec:cases}, and put $\rho(x)=\|x\|$. This exhaustion is proper because the image is closed, and positive because the orbit does not contain zero. The thresholds and the sequence $D_0$ are those constructed for the corresponding case.

\subsection{The volume form and its preservation}
\begin{lemma}\label{lem:volume-character}
Let $G$ be connected semisimple and $H\subset G$ closed reductive. A $G$-invariant nowhere-vanishing algebraic top form on $G/H$ exists if and only if
\begin{equation}\label{eq:isotropy-determinant}
 \delta_H(h):=\det\bigl(\Ad(h)|_{\g/\h}\bigr)=1
 \qquad(h\in H).
\end{equation}
The character $\delta_H$ is trivial on $H^\circ$, so connectedness of $H$ suffices.
\end{lemma}
\begin{proof}
A top covector at the base coset extends by $G$-translation precisely when it is fixed by the isotropy action. That action on top covectors has character $\delta_H^{-1}$, which gives the criterion and an algebraic form. Moreover,
\[
 \delta_H(h)=\frac{\det(\Ad(h)|_\g)}{\det(\Ad(h)|_\h)}.
\]
The numerator is one because $G$ has no nontrivial algebraic characters. On $H^\circ$ the denominator is one: a connected reductive group acts unimodularly on its Lie algebra, since its center acts trivially and its derived group has no characters. Thus $\delta_H$ factors through the finite component group.
\end{proof}

Under the hypotheses of Theorem~\ref{thm:main}, let $\omega$ be the invariant form of Lemma~\ref{lem:volume-character}.
Every fundamental vector field used above preserves $\omega$. Each variable-coefficient field has the form
\[
 W=\sum_j F_j(Q)V_{\xi_j},
\]
where $Q$ is invariant under every $V_{\xi_j}$ occurring in that field. Hence
\[
 \diver_\omega W
 =\sum_j\left(F_j(Q)\diver_\omega V_{\xi_j}
                  +V_{\xi_j}(F_j(Q))\right)=0.
\]
This applies to the unipotent shear of Lemma~\ref{lem:prepare}, the torus and unipotent flows of Lemma~\ref{lem:solvable}, and both torus flows used for self-injections.

The action factors through the effective quotient, so that reduction preserves the same form. In the group case, a left-invariant volume form on a semisimple group is also right-invariant, by unimodularity. Thus the enlarged left-right action preserves $\omega$. On the quadric, the original invariant form is preserved by the $\SL_2$ congruence action, and its two shears have invariant coefficients. Finally, the compact group action preserves $\omega$. This proves parts \textup{(a)}, \textup{(b)}, and the flow assertion of \textup{(d)} in Theorem~\ref{thm:main}. The same divergence calculation applies whenever \eqref{eq:isotropy-determinant} holds.

\begin{remark}\label{rem:conic-volume}
For the diagonal realization of $T$, the Weyl representative $\left(\begin{smallmatrix}0&-1\\1&0\end{smallmatrix}\right)$ sends $E_{12}$ to $-E_{21}$ and $E_{21}$ to $-E_{12}$ in $\mathfrak{sl}_2/\tlie$, so its isotropy determinant on $X_N=\SL_2/N(T)$ is $-1$. In fact $X_N$ has no nowhere-vanishing holomorphic two-form. On the covering quadric the residue form $\widetilde\omega=da\wedge db/a$, written on $a\ne0$, extends globally and satisfies $\tau^*\widetilde\omega=-\widetilde\omega$. A volume form downstairs would pull back to $f\widetilde\omega$ with $f$ nowhere zero and $f(-S)=-f(S)$. The quadric $\SL_2/T$ is simply connected by the homotopy sequence of $T\to\SL_2\to\SL_2/T$, so $f=e^h$ for a holomorphic $h$. Then $h(-S)-h(S)$ is a constant odd multiple of $\pi i$, whereas applying $\tau$ twice forces that constant to be zero. This is impossible. The tensor square $\widetilde\omega^{\otimes2}$ does descend, so $K_{X_N}^{\otimes2}$ is trivial.
\end{remark}

\subsection{Equivalence for every intermediate subgroup}
Under the hypotheses of Theorem~\ref{thm:main}, fix the sequence $D_0$ and thresholds constructed above, and let
\[
 \Autvfl(X)\subseteq\Gamma\subseteq\Aut(X).
\]
By parts \textup{(a)}, \textup{(b)}, and \textup{(d)}, the maps that normalize sparse sequences and extend injections $D_0\to D_0$ belong to $\Gamma$.
If $D=\{x_m\}$ is weakly $\Gamma$-tame, choose $\Psi\in\Gamma$ with
\[
 \rho(\Psi(x_m))\ge R_m+1>R_m\qquad(m\ge2).
\]
By parts \textup{(a)} and \textup{(d)}, there is a composition $\Theta$ of volume-preserving complete flow maps with $\Theta\Psi(x_m)=p_m$. Thus every enumeration of a weakly $\Gamma$-tame set can be sent to $p_m$ in its given order.

If $D$ is $\Gamma$-equivalent to $D_0$, conjugating the extensions from Theorem~\ref{thm:main}\textup{(b)} by that equivalence proves strong $\Gamma$-tameness.

Conversely, strong $\Gamma$-tameness implies weak $\Gamma$-tameness on any manifold. For an exhaustion and bounds, each compact sublevel contains only finitely many points of a closed discrete set. Recursively choose distinct image points within the set beyond the requested bounds. Strong tameness extends this self-injection. This proves the equivalence in part \textup{(c)} and its refinement in part \textup{(d)}.

Let $f:D\to E$ be an injection between two weakly $\Gamma$-tame sets. Its image is weakly $\Gamma$-tame: extend any assigned bounds on that subset to bounds on $E$. Enumerate $D=(x_m)$ and its image by $y_m=f(x_m)$. Normalize these two sequences separately, both to $p_m$, using maps $\Phi_D,\Phi_f\in\Gamma$. Then
\[
 \Phi_f^{-1}\Phi_D(x_m)=y_m=f(x_m).
\]
This proves Theorem~\ref{thm:main}.

\subsection{The four RR axioms}
In Winkelmann's definition \cite[Definition~2.1]{WinkelmannRR}, an RR-structure is a collection $\cT$ of infinite closed discrete subsets satisfying:
\begin{enumerate}[label=\textup{(RR\arabic*)},leftmargin=3.8em]
\item every injection between two members of $\cT$ extends to an automorphism;
\item for some proper exhaustion $\rho$ and positive $r_n\uparrow\infty$, every infinite $D\subset X$ with $\#\{x\in D:\rho(x)\le r_n\}\le n$ for all $n$ belongs to $\cT$;
\item $\cT$ is invariant under $\Aut(X)$;
\item there are a nonzero complete holomorphic vector field $W$ and a nonconstant holomorphic function $f$ with $W(f)=0$.
\end{enumerate}
Take $\Gamma=\Aut(X)$ and put $\cT=\Aut(X)\cdot D_0$. Injective interpolation was proved above, and automorphism invariance follows from this definition.

For the threshold axiom, take the counting radii $r_n=R_{n+1}$. If an infinite subset $D\subset X$ satisfies
\[
 \#\{x\in D:\rho(x)\le r_n\}\le n\qquad(n\ge1),
\]
it meets every compact set in finitely many points, because the $r_n$ are unbounded and $\rho$ is proper. It is therefore closed discrete. Enumerate it in nondecreasing order of $\rho$. For $m\ge2$,
\[
 \rho(x_m)>r_{m-1}=R_m.
\]
For $m=1$, positivity of $\rho$ gives $\rho(x_1)>0=R_1$. Theorem~\ref{thm:main}\textup{(a)} therefore normalizes the sequence, so $D\in\cT$.

For the final RR axiom under the hypotheses of Proposition~\ref{prop:principal}, choose a nonzero positive-root vector $\xi\in\ulie$. Effectivity makes its action field nonzero. It is complete and annihilates the nonconstant $U$-invariant function $q$. For a group quotient use the effective left-right presentation and the same argument. On the quadric use the upper field in \eqref{eq:upper2} with first integral $c$. The field is nonzero, for example at the identity matrix, and $c$ is nonconstant. Thus all four axioms in \cite[Definition~2.1]{WinkelmannRR} hold. This proves Theorem~\ref{thm:RR}.

\begin{proof}[Proof of Corollary~\ref{cor:oka}]
Apply the Oka-complement theorem for RR-spaces, \cite[Theorem~2.2(iii), Proposition~5.25]{WinkelmannRR}, to Theorem~\ref{thm:RR}.
\end{proof}

Every unbounded subset of $X$ contains a strongly $\Autvfl(X)$-tame sequence: choose distinct points successively beyond $R_m$ and apply Theorem~\ref{thm:main}\textup{(a), (b), (d)}.

\subsection{Normal form on an algebraic affine line}
\begin{corollary}[Normal form on an algebraic affine line]\label{cor:line-normal-form}
Under the hypotheses of Theorem~\ref{thm:RR}, let $C\subset X$ be a closed
algebraic curve isomorphic to $\mathbb A^1$, and fix an infinite closed
discrete subset $D_C\subset C$. For an infinite closed discrete subset
$D\subset X$, the following are equivalent:
\begin{enumerate}[label=\textup{(\roman*)}]
\item $D$ is tame;
\item some $\Phi\in\Aut(X)$ satisfies $\Phi(D)\subset C$;
\item some $\Phi\in\Aut(X)$ satisfies $\Phi(D)=D_C$.
\end{enumerate}
Such curves $C$ exist in every positive-dimensional $X$ covered by the theorem.
\end{corollary}
\begin{proof}
The variety $X$ is smooth and affine. Nilpotent elements span $\g$, and their action fields span every tangent space of $X$. These fields generate algebraic additive-group actions, so $X$ is flexible. Theorem~\ref{thm:RR} establishes its RR property. Thus \cite[Theorem~6.9]{WinkelmannRR} applies to $X$, the curve $C$, and the set $D_C$, and gives the three equivalent conditions. At least one root subgroup acts nontrivially,
since $\dim X>0$. Its nontrivial orbit is closed in the affine variety
and isomorphic to $\mathbb A^1$; see also the proof of
\cite[Corollary~6.8]{WinkelmannRR}.
\end{proof}

Theorem~\ref{thm:main}\textup{(a)} bounds the number of flow maps sending a sufficiently sparse enumerated sequence to $D_0$. Corollary~\ref{cor:line-normal-form} places no such bound on an automorphism sending $D$ onto $D_C$.

\subsection{Extension to disconnected isotropy}

\begin{proposition}\label{prop:disconnected}
Let $G$ be a connected complex semisimple algebraic group and $H\subset G$ a closed reductive subgroup. Put $N=\bigcap_{g\in G}gHg^{-1}$. Suppose $\dim G/H>0$ and $\dim(H/N)>0$. Then $X=G/H$ is an RR-space, and the conclusions of Theorem~\ref{thm:main}\textup{(a)--(c)} hold. The Oka-complement and affine-line conclusions of Corollaries~\ref{cor:oka} and~\ref{cor:line-normal-form} hold as well.

In particular, $\SL_2/N(T)\simeq\mathbf P^2\setminus \mathcal Q$ is an RR-space, and the conclusion of Theorem~\ref{thm:main}\textup{(a)} holds with $n\le3$.

If $\delta_H=1$, the volume-preserving refinement in Theorem~\ref{thm:main}\textup{(d)} also holds.
\end{proposition}
\begin{proof}
Replace $G,H$ by $G/N,H/N$. The effective group is semisimple, the isotropy is reductive and positive-dimensional, and the variety is unchanged. In rank at least two, Proposition~\ref{prop:principal} applies. Its proof uses $\Lie(H)$ in the transversality step and imposes no connectedness condition. The compact action is itself a complete flow map, so $n\le6$ in part \textup{(a)}.

In rank one, the inverse image of the isotropy in $\SL_2$ is $T$ or $N(T)$, as in Section~\ref{subsec:antipodal}. Sections~\ref{subsec:quadric} and~\ref{subsec:antipodal} normalize sufficiently sparse sequences by three flows and extend injective self-maps of the fixed sequence by two flows. In all cases these maps lie in $\Autfl(X)$. The subgroup and threshold arguments above therefore apply. The last RR axiom uses a nonzero root field with first integral $q$ in the principal case, and the upper field with first integral $c$ or $c^2$ in rank one. The Oka-complement result follows from the RR property; the affine-line argument applies because these homogeneous spaces are flexible. Finally, if $\delta_H=1$, Lemma~\ref{lem:volume-character} supplies the invariant form and the divergence calculation applies to every constructed field. The form and the action are unchanged by passage to the effective quotient.
\end{proof}

The left-right reduction in Section~\ref{sec:cases} requires trivial effective isotropy and therefore does not apply when $H/N$ is finite and nontrivial.

\section{Spherical quotients}\label{sec:spherical}

Throughout this section $G$ is simple of rank at least two. For a sufficiently sparse sequence in a positive-dimensional closed orbit in an irreducible module, one compact group element makes the values of a highest-weight linear function distinct and discrete. If a Borel orbit is open, three further flows send the sequence to fixed points $p_m$.

\begin{theorem}\label{thm:sparse}
Let $V$ be a nontrivial finite-dimensional irreducible rational $G$-module, and let $X\subset V$ be a positive-dimensional closed $G$-orbit. Suppose a Borel subgroup has an open orbit on $X$. Fix a compact real form $K$ and a $K$-invariant Hermitian norm on $V$.

There are a closed discrete sequence $p_m\in X$ and thresholds $R_1=0<R_2<\cdots\to\infty$ with the following property. For every sequence of distinct points $x_m$ satisfying $\|x_m\|>R_m$ for $m\ge2$, there are flow maps $\Phi_1,\Phi_2,\Phi_3$ and $k\in K$ such that
\[
 (\Phi_3\circ\Phi_2\circ\Phi_1)(kx_m)=p_m\qquad(m\ge1).
\]
Every injective self-map of $\{p_m\}$ is realized by two complete flows. For every subgroup $\Autfl(X)\subseteq\Gamma\subseteq\Aut(X)$, the conclusions of Theorem~\ref{thm:main}(c) hold with this choice of $D_0$.

If $X$ admits a $G$-invariant algebraic volume form $\omega$, all these flows preserve it, and the same conclusions hold for every subgroup containing $\Autvfl(X)$. Existence of the form is characterized by Lemma~\ref{lem:volume-character}.
\end{theorem}

Fix $B=TU$. We use highest-weight theory as in \cite{Borel,Humphreys}.

\subsection{The two linear functions}

\begin{lemma}\label{lem:twofunctions}
There are nonzero linear forms $q,r\in V^*$, characters $\chi,\psi$ of $T$ with independent differentials, and a point $o$ in the open Borel orbit such that
\begin{equation}\label{eq:twofunctions}
 q(ux)=q(x),\qquad q(tx)=\chi(t)q(x),\qquad
 r(tx)=\psi(t)r(x),\qquad q(o)=r(o)=1.
\end{equation}
Here $u\in U$, $t\in T$, and the first three identities hold on all of $V$.
\end{lemma}
\begin{proof}
Choose a nonzero highest-weight vector $q\in(V^*)^U$ for the contragredient action. This one-dimensional space is $T$-stable. With convention \eqref{eq:weightconvention}, $q$ transforms by a character $\chi$. Its differential is nonzero: a highest-weight vector of weight zero in an irreducible module would make that module trivial.

The weights of $V^*$ span $\tlie^*$. Indeed, the differential of the nontrivial representation of the simple Lie algebra $\g$ has zero kernel. An element of $\tlie$ annihilated by all weights would act as zero on $V^*$ and hence be zero. Since $\dim T\ge2$, a weight not proportional to the weight of $q$ occurs. Choose an eigenvector $r$ of that weight and write $r(tx)=\psi(t)r(x)$.

The orbit $X$ spans $V$: its span is a nonzero $G$-submodule. Hence neither $q$ nor $r$ vanishes identically on $X$. The homogeneous space $X$ is irreducible, and its open Borel orbit is dense. That orbit therefore meets $\{q\ne0\}\cap\{r\ne0\}$. Choose $o$ in the intersection and rescale the two linear forms so that $q(o)=r(o)=1$.
\end{proof}

Write
\begin{equation}\label{eq:borelcell}
 \Omega=B\cdot o=UT\cdot o.
\end{equation}
The function $q$ does not vanish on $\Omega$, since it is $U$-invariant and transforms by a character under $T$. The choice of $o$ ensures that $r$ is nonzero on the torus orbit used below.

\subsection{The fixed sequence}
By independence, choose $A_0\in\tlie$ satisfying
\begin{equation}\label{eq:A0}
 d\chi(A_0)=d\psi(A_0)=1.
\end{equation}
For $m\in\N$ let
\begin{equation}\label{eq:model}
 s_m=\exp((\log m)A_0),\qquad p_m=s_mo,
 \qquad D_0=\{p_m:m\in\N\}.
\end{equation}
Equations \eqref{eq:twofunctions} and \eqref{eq:A0} give
\begin{equation}\label{eq:modellabels}
 q(p_m)=r(p_m)=m.
\end{equation}
The equality $q(p_m)=m$ shows that the points are distinct and have no accumulation point in $X$, so they form a closed discrete subset.

\subsection{Sequences in the open Borel orbit}

\begin{proposition}\label{prop:positioned}
Let $x_m\in\Omega$ be distinct, and assume that $q(x_m)$ are pairwise distinct and form a closed discrete subset of $\C$. Then three complete holomorphic flow maps carry $x_m$ to $p_m$ for every $m$.
\end{proposition}
\begin{proof}
Choose, independently for each $m$, a factorization
\begin{equation}\label{eq:factorpoints}
 x_m=u_mt_mo,\qquad u_m\in U,\quad t_m\in T.
\end{equation}

Since $q$ is $U$-invariant, Corollary~\ref{cor:unipotent} provides a complete flow with time-one map $\Phi_1$ satisfying
\begin{equation}\label{eq:Uoutput}
 \Phi_1(x_m)=u_m^{-1}x_m=t_mo=:a_m.
\end{equation}
Explicitly, choose an entire map $F:\C\to\ulie$ with
\[
 F(q(x_m))=\log(u_m^{-1}),
\]
and take $\Phi_1(x)=\exp(F(q(x)))x$. It preserves $q$, so the values $q(a_m)=q(x_m)$ are still pairwise distinct and form a closed discrete subset of $\C$.

Both $a_m=t_mo$ and $p_m=s_mo$ lie in $T\cdot o$. The values $r(p_m)=m$ are also pairwise distinct and form a closed discrete subset of $\C$. Proposition~\ref{prop:torusmatch} now gives two complete flow maps $\Phi_2,\Phi_3$ satisfying
\[
 \Phi_3\Phi_2(a_m)=p_m.
\]
Together with \eqref{eq:Uoutput}, this proves the proposition.
\end{proof}

Let $a'_m$ be the intermediate points after the first torus flow in Proposition~\ref{prop:torusmatch}, so $r(a'_m)=r(p_m)=m$. The three flows use the following invariant functions:
\[
\begin{array}{c|c|c|c}
\text{points of }X&\text{acting group}&\text{invariant function}&\text{interpolate at}\\ \hline
x_m=u_mt_mo&U&q&q(x_m)\\
a_m=t_mo&T_q&q&q(a_m)=q(x_m)\\
a'_m&T_r&r&r(a'_m)=r(p_m)=m.
\end{array}
\]

\begin{proof}[Proof of the normalization assertion in Theorem~\ref{thm:sparse}]
The closed positive-dimensional orbit $X$ does not contain zero, since zero is a fixed point. Apply Proposition~\ref{prop:position} and Corollary~\ref{cor:irreduciblespread} to $q$, with $\Omega_m=\Omega$ for every $m$. They give thresholds and one $k\in K$ for which $kx_m\in\Omega$, the values $q(kx_m)$ are nonzero and pairwise distinct, and their moduli tend to infinity. Proposition~\ref{prop:positioned} supplies the three remaining flows.
\end{proof}

\subsection{Self-injections of the fixed sequence}

For an injection $\sigma:\N\to\N$, apply Proposition~\ref{prop:torusmatch} to $a_m=p_m$ and $b_m=p_{\sigma(m)}$. By \eqref{eq:modellabels}, $q(a_m)=m$ and $r(b_m)=\sigma(m)$. Both sequences of numbers are pairwise distinct and closed discrete, so two complete torus flows realize the self-injection.

The compact action is itself one complete flow. The proof of Theorem~\ref{thm:main}(c) uses normalization of sufficiently sparse sequences and extension of injective self-maps of the fixed sequence inside the chosen subgroup, so it applies here to every $\Gamma$ containing $\Autfl(X)$. If an invariant volume form $\omega$ exists, the divergence calculation in Section~\ref{sec:tameness} applies to all these flows, and the same argument works for every $\Gamma$ containing $\Autvfl(X)$. This completes the proof of Theorem~\ref{thm:sparse}.

\begin{corollary}\label{cor:sphericalRR}
Every $X$ in Theorem~\ref{thm:sparse} is an RR-space.
\end{corollary}
\begin{proof}
The interpolation, counting, and invariance arguments in Section~\ref{sec:tameness} apply to this sequence and these thresholds. For the fourth axiom, choose a nonzero positive-root vector $\xi\in\ulie$. Its action field annihilates $q$ and is nonzero on $X$: otherwise $\xi$ would annihilate the linear span $V$ of $X$, contradicting faithfulness of the differential of a nontrivial representation of the simple Lie algebra $\g$. The function $q$ is nonconstant because $q(p_m)=m$.
\end{proof}

\subsection{Losev's embedding criterion}

For a spherical quotient, Losev's theorem gives an intrinsic criterion for the embedding used in Theorem~\ref{thm:sparse}. Let $G$ be connected reductive and $H\subset G$ reductive. Write $\Aut_G(G/H)$ for the group of automorphisms commuting with the $G$-action. Such an automorphism is determined by the image of the base coset. With the convention
\[
 nH\longmapsto\bigl(gH\longmapsto gn^{-1}H\bigr),
\]
we obtain a group isomorphism $N_G(H)/H\simeq\Aut_G(G/H)$.

Losev uses a subgroup of this equivariant automorphism group. Write $\mathfrak A_G(X)$ for the \emph{central equivariant automorphisms}: these act by a scalar on every space of rational $B$-semi-invariant functions of a fixed weight. His criterion \cite[Theorem~1.4]{Losev} states that
\begin{equation}\label{eq:Losev}
 \begin{gathered}
 G/H\text{ embeds equivariantly as a closed subvariety}\\
 \text{of an irreducible }G\text{-module}\\
 \Longleftrightarrow\quad \mathfrak A_G(G/H)\text{ is finite cyclic}.
 \end{gathered}
\end{equation}
\begin{lemma}\label{lem:sphericalcentral}
If $H$ is spherical, then every $G$-equivariant automorphism of $G/H$ is central. In particular,
\[
 \mathfrak A_G(G/H)=\Aut_G(G/H)\simeq N_G(H)/H.
\]
\end{lemma}
\begin{proof}
Two nonzero rational $B$-semi-invariants of the same weight have a $B$-invariant ratio. A rational $B$-invariant function is constant on the open $B$-orbit and hence constant in the function field. Thus every nonzero rational weight space is one-dimensional.

A $G$-equivariant automorphism commutes with $B$ and preserves each such weight space. On a one-dimensional space it acts by a scalar. This specialization is also recorded in \cite[Example~5.1]{Losev}.
\end{proof}

\begin{proof}[Proof of Corollary~\ref{cor:spherical}]
The finite cyclic hypothesis and Lemma~\ref{lem:sphericalcentral} give a closed equivariant embedding $G/H\hookrightarrow V$ into an irreducible rational module by \eqref{eq:Losev}. The module is nontrivial because $\dim G/H>0$. Sphericity gives an open Borel orbit. Theorem~\ref{thm:sparse} applies. Since the compact action is itself a flow map, its displayed composition gives $n\le4$.
\end{proof}

Finite cyclicity in \eqref{eq:Losev} characterizes the closed embedding into an irreducible module. Theorem~\ref{thm:RR} proves the RR property for connected reductive $H$ regardless of this criterion.

For a nonspherical quotient, $\mathfrak A_G(G/H)$ can differ from the full normalizer quotient. For example, a regular semisimple $\SL_3$-orbit embeds as a closed subvariety of the irreducible adjoint module, and $N_{\SL_3}(T)/T\simeq S_3$. Here $\dim\SL_3/T=6>5=\dim B$, so the quotient is not spherical.

\section{A complete calculation on symmetric \texorpdfstring{$3\times3$}{3 by 3} matrices}\label{sec:matrices}

We give explicit formulas for the flow maps in Theorem~\ref{thm:sparse} when $X=\SL_3/\SO_3$. Set
\[
 X=\{S\in\Sym_3(\C):\det S=1\},\qquad
 S=\begin{pmatrix}a&b&c\\b&d&e\\c&e&f\end{pmatrix}.
\]
The group $\SL_3$ acts by $g\cdot S=gSg^{\mathsf T}$. Take the upper-triangular Borel and its unipotent radical, and put
\[
 q(S)=f,\qquad r(S)=d,\qquad \Delta(S)=df-e^2.
\]
For $t=\diag(t_1,t_2,t_3)$ with $t_1t_2t_3=1$, the characters are $\chi(t)=t_3^2$ and $\psi(t)=t_2^2$. They are independent. The identity matrix lies in the open Borel orbit and has $q(I)=r(I)=1$. The choice
\[
 A_0=\diag(-1,1/2,1/2)
\]
gives the fixed sequence
\[
 p_m=\diag(m^{-2},m,m).
\]

\subsection{The compact congruence}
Equip $\Sym_3(\C)$ with an $\SU_3$-invariant Hermitian norm. Apply Proposition~\ref{prop:position} and Corollary~\ref{cor:irreduciblespread} to $q(S)=S_{33}$ and the nonempty open set $\{f\Delta\ne0\}$. Given a sufficiently sparse input sequence $X_m$, choose $k\in\SU_3$ as in that proposition and write $S_m=kX_mk^{\mathsf T}$, $f_m=f(S_m)$, and $\Delta_m=\Delta(S_m)$. Then $f_m\Delta_m\ne0$ for every $m\ge1$, and the $f_m$ are pairwise distinct and closed discrete in $\C$. Write $k=\exp A$ with $A\in\mathfrak{su}_3$. Define
\[
 \Phi^{(0)}_t(S)=e^{tA}S e^{tA^{\mathsf T}},\qquad t\in\C.
\]
This is a complete holomorphic flow with $\Phi^{(0)}_1(X_m)=S_m$.

\subsection{Removing the unipotent part}
The open Borel orbit is the locus $f\Delta\ne0$. To see the elimination explicitly, write $u_{ij}(t)=I+tE_{ij}$ and set
\begin{equation}\label{eq:matrixparameters}
 y=-\frac ef,\qquad z=-\frac cf,\qquad
 x=-\frac{bf-ce}{\Delta}.
\end{equation}
The congruences $u_{23}(y)$, then $u_{13}(z)$, then $u_{12}(x)$ send $S$ to
\begin{equation}\label{eq:matrixdiag}
 \diag\left(\frac1\Delta,\frac\Delta f,f\right).
\end{equation}
After the first two operations, the upper-left block is
\[
 \begin{pmatrix}
 a-c^2/f&b-ce/f\\
 b-ce/f&d-e^2/f
 \end{pmatrix}.
\]
Its second diagonal entry is $\Delta/f$, so the final operation removes its off-diagonal entry. Since the determinant remains one, the remaining diagonal entry is $1/\Delta$. Conversely, every diagonal matrix of determinant one with nonzero diagonal entries is $T$-congruent to $I$. This proves the open-orbit assertion.

The product acting on $S$ is
\[
 v(S)=u_{12}(x)u_{13}(z)u_{23}(y)
 =\begin{pmatrix}1&x&z+xy\\0&1&y\\0&0&1\end{pmatrix}.
\]
Its nilpotent logarithm is
\begin{equation}\label{eq:matrixlog}
 \log v(S)=xE_{12}+yE_{23}+(z+xy/2)E_{13}.
\end{equation}
Evaluating \eqref{eq:matrixparameters} at $S_m$ gives numbers $x_m,y_m,z_m$. Interpolate these data by an entire strictly upper-triangular matrix function $L:\C\to\ulie$ with
\[
 L(f_m)=x_mE_{12}+y_mE_{23}
       +\left(z_m+\frac{x_my_m}{2}\right)E_{13}.
\]
The entire function $L$ defines a flow on $X$ by
\[
 \Phi^{(1)}_t(S)=\exp\bigl(tL(S_{33})\bigr)S
                  \exp\bigl(tL(S_{33})\bigr)^{\mathsf T}.
\]
It is complete by Lemma~\ref{lem:groupflow}, since every upper-unitriangular congruence preserves $S_{33}$. At time one it sends each $S_m$ to the diagonal matrix \eqref{eq:matrixdiag}, with $f=f_m$ and $\Delta=\Delta_m$.

\subsection{The two torus operations}
Starting with \eqref{eq:matrixdiag}, use
\[
 H_1=\diag(-1,1,0).
\]
Its congruence flow preserves $q=f$. Choose an entire $F_1$ with
\[
 e^{2F_1(f_m)}=\frac{mf_m}{\Delta_m}.
\]
The corresponding flow is
\[
 \Phi^{(2)}_t(S)=e^{tF_1(S_{33})H_1}S
                 e^{tF_1(S_{33})H_1^{\mathsf T}}.
\]
At time one it sends the $m$th diagonal matrix to
\[
 \diag\left(\frac1{mf_m},m,f_m\right).
\]
At this point $r=S_{22}$ has value $m$. The congruence action generated by
\[
 H_2=\diag(-1,0,1)
\]
preserves $r$ on all of $X$. Choose an entire $F_2$ with
\[
 e^{2F_2(m)}=\frac m{f_m}.
\]
The corresponding flow is
\[
 \Phi^{(3)}_t(S)=e^{tF_2(S_{22})H_2}S
                 e^{tF_2(S_{22})H_2^{\mathsf T}}.
\]
At time one it sends the last diagonal entry of each sequence point to $m$. Determinant one then forces the first entry to be $m^{-2}$. Thus
\[
 \Phi^{(3)}_1\circ\Phi^{(2)}_1\circ\Phi^{(1)}_1\circ\Phi^{(0)}_1
\]
normalizes the original sequence to $p_m$. The upper-unipotent flow and the $H_1$-flow preserve $f$, while the $H_2$-flow preserves $r$. Their coefficient functions are therefore invariant along the corresponding flows, and the divergence calculation of Section~\ref{sec:tameness} proves volume preservation. Interpolating the whole matrix logarithm realizes the three elementary unipotent operations by one flow.

\begin{remark}
The map
\[
 S\longmapsto\diag(e^{-f},1,e^f)\,S\,\diag(e^{-f},1,e^f)
\]
is not an automorphism. On $S(t)=\diag(t^{-1},1,t)$, $t\ne0$, it sends $S(t)$ to $S(te^{2t})$. The derivative of this restricted map vanishes at $t=-1/2$.
\end{remark}

For all $n\ge3$, the space $\SL_n/\SO_n$ satisfies Corollary~\ref{cor:spherical}: three flows suffice after the compact action, independently of $n$. Section~\ref{sec:cases} gives two explicit shears after the compact action for $n=2$.

\section{Examples and normalizer quotients}\label{sec:examples}

We specify the group forms and subgroup embeddings because the normalizer quotients depend on them. We use $\Sp_{2a}$ for the symplectic group on a $2a$-dimensional vector space, $\mu_a$ for the group of $a$-th roots of unity, and $C_a$ for an abstract cyclic group of order $a$. All subgroups in the tables are connected. Classical block embeddings preserve the displayed decompositions and forms. Every positive-dimensional quotient in this section is an RR-space by Theorem~\ref{thm:RR}; the calculations determine whether Corollary~\ref{cor:spherical} also applies.

The sphericity of the pairs below is part of the classifications of Kr\"amer and Knop--R\"ohrle \cite{Kraemer,KnopRohrle}; the latter's Tables~1 and~2 include every family used here. We calculate the normalizer quotients directly.

\begin{proposition}\label{prop:examples}
Every pair in Table~\ref{tab:positive} satisfies Corollary~\ref{cor:spherical}.
\end{proposition}

\begin{table}[ht]
\centering
\small
\renewcommand{\arraystretch}{1.25}
\begin{tabular}{@{}lll@{}}
\toprule
$G$ and parameter range & $H$ & $N_G(H)/H$\\
\midrule
$\SL_n$, $n\ge3$ & $\SO_n$ & $\mu_n$\\
$\SL_{2a}$, $a\ge2$ & $\Sp_{2a}$ & $\mu_a$\\
$\SL_{p+q}$, $p,q\ge1$, $p+q\ge3$
 & $S(\GL_p\times\GL_q)$ & $1$ if $p\ne q$; $C_2$ if $p=q$\\
$\Sp_{2p+2q}$, $p,q\ge1$
 & $\Sp_{2p}\times\Sp_{2q}$ & $1$ if $p\ne q$; $C_2$ if $p=q$\\
$\SO_n$, $n\ge5$ & $\SO_{n-1}$ & $C_2$\\
$F_4$ & $\Spin_9$ & $1$\\
\bottomrule
\end{tabular}
\caption{Spherical pairs with a closed equivariant embedding into an irreducible module.}\label{tab:positive}
\end{table}

\subsection{Symmetric and alternating forms}
Consider first $\SO_n\subset\SL_n$, with $n\ge3$. The natural $\SO_n$-module is irreducible and has a unique invariant symmetric form up to scalar. An element $g\in\SL_n$ normalizing $\SO_n$ must therefore be a similitude:
\[
 g^{\mathsf T}g=cI.
\]
Taking determinants gives $c^n=1$. The multiplier is a homomorphism from the normalizer to $\mu_n$ with kernel $\SO_n$. It is surjective. To realize $c\in\mu_n$, choose $a$ with $a^2=c$. Then $a^n=\pm1$, and there is an orthogonal matrix $O$ with determinant $a^{-n}$. The matrix $aO$ has determinant one and multiplier $c$. Thus
\begin{equation}\label{eq:SOnormal}
 N_{\SL_n}(\SO_n)/\SO_n\simeq\mu_n.
\end{equation}

Here a closed embedding is visible without Losev's theorem:
\[
 \SL_n/\SO_n\simeq
 \{S\in\Sym_n(\C):\det S=1\},
 \qquad g\cdot S=gSg^{\mathsf T}.
\]
The ambient symmetric-square module is irreducible. Transitivity follows by factoring a nonsingular symmetric matrix as $gg^{\mathsf T}$ and adjusting the sign of $\det g$ by an orthogonal matrix. The stabilizer of $I$ is $\SO_n$. Section~\ref{sec:matrices} writes the normalizing flows explicitly for $n=3$.

For $\Sp_{2a}\subset\SL_{2a}$, normalize a nondegenerate alternating matrix $J$ by $\Pf(J)=1$. A normalizing element satisfies
\[
 g^{\mathsf T}Jg=cJ.
\]
The Pfaffian identity gives $\det g=c^a$. Thus the multiplier lies in $\mu_a$. Conversely, given $c^a=1$, a square root $b^2=c$ yields $bI\in\SL_{2a}$ with multiplier $c$. The kernel is $\Sp_{2a}$, so
\[
 N_{\SL_{2a}}(\Sp_{2a})/\Sp_{2a}\simeq\mu_a.
\]
The corresponding closed realization is the hypersurface $\Pf(S)=1$ of nondegenerate alternating matrices, inside the irreducible exterior-square module. The range $a\ge2$ excludes the zero-dimensional quotient $\SL_2/\Sp_2$.

\subsection{Two-block quotients}
For
\[
 H=S(\GL_p\times\GL_q)\subset\SL_{p+q},
\]
a general element of the connected center of $H$ has the two blocks as its eigenspaces. An element normalizing $H$ preserves these subspaces or exchanges them. If their dimensions differ, it must preserve both and hence belongs to $H$. If $p=q$, an exchange exists inside $\SL_{2p}$ after adjusting its determinant by a block-diagonal matrix. The block-preserving subgroup is still $H$, so the quotient has order two.

For $H=\Sp_{2p}\times\Sp_{2q}\subset\Sp_{2p+2q}$, the natural module is the sum of two inequivalent irreducible $H$-modules, each nondegenerate for the symplectic form. An element normalizing $H$ permutes these two summands. Unequal dimensions prevent an exchange. An element preserving the two blocks lies in their symplectic product. For equal dimensions the block exchange is symplectic, giving the additional coset. This proves the two block-family entries of Table~\ref{tab:positive}.

\subsection{Quadrics and an exceptional pair}
The subgroup $\SO_{n-1}\subset\SO_n$ fixes a nonisotropic vector $v$. Its fixed-vector space is precisely $\C v$, so the normalizer preserves that line. Such an orthogonal transformation acts on the line by $1$ or $-1$. Those acting by $1$ form $H$, and the second possibility is realized by combining $v\mapsto-v$ with a reflection on $v^\perp$. Therefore the normalizer quotient is $C_2$. The homogeneous space is the affine quadric of vectors with the same nonzero norm as $v$, a closed orbit in the natural irreducible module.

For the standard subgroup $H=\Spin_9\subset F_4$, the subgroup has full rank and its Dynkin diagram has no nontrivial automorphism. Every algebraic automorphism of $H$ is therefore inner. If $g\in N_{F_4}(H)$, multiply it by an element of $H$ so that it centralizes $H$. It then centralizes a maximal torus of $H$, which is also a maximal torus of $F_4$. The centralizer of that torus in $F_4$ is the torus itself, contained in $H$. Thus $N_{F_4}(H)=H$.

These calculations prove Proposition~\ref{prop:examples}.

\subsection{Pairs without a closed embedding into an irreducible module}
Table~\ref{tab:negative} contains spherical pairs for which the closed-embedding criterion fails.

\begin{table}[ht]
\centering
\small
\renewcommand{\arraystretch}{1.25}
\begin{tabular}{@{}lll@{}}
\toprule
$G$ and parameter range & $H$ & $N_G(H)/H$\\
\midrule
$\SL_{p+q}$, $p>q\ge1$ & $\SL_p\times\SL_q$ & $\C^*$\\
$\SL_{2a+1}$, $a\ge2$ & $\Sp_{2a}$ & $\C^*$\\
$\SO_{4a}$, $a\ge2$ & $\SO_{2a}\times\SO_{2a}$ & $C_2\times C_2$\\
\bottomrule
\end{tabular}
\caption{Spherical pairs without a closed equivariant embedding into an irreducible module.}\label{tab:negative}
\end{table}

For the first row, the normalizer of $\SL_p\times\SL_q$ is $S(\GL_p\times\GL_q)$. The natural module has two inequivalent isotypic summands, and their different dimensions prevent a permutation. This also covers $q=1$, where the second summand is the fixed line and $\SL_1$ denotes the trivial group. The homomorphism
\[
 \diag(A,B)\longmapsto\det A
\]
is onto $\C^*$ and has kernel $\SL_p\times\SL_q$. Hence removing the central torus from the two-block Levi subgroup changes a trivial normalizer quotient into a one-dimensional torus. For example, the closed embedding into an irreducible module exists for
\[
 \SL_5/S(\GL_3\times\GL_2),
\]
and fails for $\SL_5/(\SL_3\times\SL_2)$. Both quotients satisfy Theorem~\ref{thm:main}.

For $\Sp_{2a}\subset\SL_{2a+1}$, the natural module is the symplectic representation plus a fixed line. Every element of the normalizer preserves this decomposition. Its elements have the form
\[
 \diag(A,(\det A)^{-1}),\qquad A\in\mathrm{GSp}_{2a},
\]
where $\mathrm{GSp}_{2a}$ is the symplectic similitude group. The multiplier identifies the quotient by $\Sp_{2a}$ with $\C^*$. This is also the example discussed by Losev \cite[Example~5.1]{Losev}. His theorem for embeddings that need not be closed allows a one-dimensional torus, while the closed-embedding theorem requires finite cyclicity.

\subsection{A finite noncyclic quotient and the parity of the block size}
Let $m\ge3$, and put
\[
 G=\SO_{2m},\qquad H=\SO_m\times\SO_m
\]
with the standard orthogonal block decomposition. Its block-preserving normalizer is $S(\Ort_m\times\Ort_m)$. Modulo $H$, this subgroup has order two, represented by
\[
 A=\diag(R,R),
\]
where $R$ is an orthogonal reflection with $R^2=I$. An element of the normalizer either preserves the two irreducible summands or exchanges them, so the full normalizer quotient has four elements.

If $m$ is even, the exchange
\[
 J=\begin{pmatrix}0&I_m\\ I_m&0\end{pmatrix}
\]
has determinant one and satisfies $J^2=I$. It commutes with $A$, and their cosets are distinct nontrivial involutions. Thus
\begin{equation}\label{eq:orthogonaleven}
 N_{\SO_{2m}}(\SO_m\times\SO_m)/H\simeq C_2\times C_2
 \qquad(m\text{ even}).
\end{equation}
This proves the last row of Table~\ref{tab:negative}.

If $m$ is odd, take instead
\[
 J=\begin{pmatrix}0&I_m\\-I_m&0\end{pmatrix}.
\]
It is orthogonal with determinant one, and $J^2=-I_{2m}$. Since $\det(-I_m)=-1$, this square belongs to the nontrivial block-preserving coset, not to $H$. Therefore
\begin{equation}\label{eq:orthogonalodd}
 N_{\SO_{2m}}(\SO_m\times\SO_m)/H\simeq C_4
 \qquad(m\text{ odd}).
\end{equation}
The odd-block pair therefore satisfies Corollary~\ref{cor:spherical}. In particular, $\SO_8/(\SO_4\times\SO_4)$ has no closed equivariant embedding into an irreducible module, whereas $\SO_{10}/(\SO_5\times\SO_5)$ admits one.

\section*{Acknowledgment of assistance}
OpenAI's ChatGPT assisted with proof exploration, literature searches, checks of algebraic identities, and preparation and revision of the manuscript. The author takes full responsibility for all statements and proofs.


\begin{thebibliography}{99}

\bibitem{AndristArzhantsev}
R.~B. Andrist and I.~Arzhantsev,
\emph{Infinite transitivity and polynomial vector fields},
preprint (2026), \href{https://arxiv.org/abs/2605.20480}{arXiv:2605.20480}.

\bibitem{AndristDraismaEtAl}
R.~B. Andrist, J.~Draisma, G.~Freudenburg, G.~Huang, and F.~Kutzschebauch,
\emph{A criterion for the algebraic density property of smooth affine $\SL_2$-varieties},
preprint (2025), revised 2026,
\href{https://arxiv.org/abs/2502.13903v2}{arXiv:2502.13903v2}.

\bibitem{AndristUgolini}
R.~B. Andrist and R.~Ugolini,
\emph{A new notion of tameness},
J. Math. Anal. Appl. \textbf{472} (2019), 196--215.
\href{https://doi.org/10.1016/j.jmaa.2018.11.018}{doi:10.1016/j.jmaa.2018.11.018}.

\bibitem{AndristUgoliniHom}
R.~B. Andrist and R.~Ugolini,
\emph{Tame sets in homogeneous spaces},
Transform. Groups \textbf{30} (2025), 1--18.
\href{https://doi.org/10.1007/s00031-022-09781-1}{doi:10.1007/s00031-022-09781-1}.

\bibitem{Arzhantsev}
I.~V. Arzhantsev,
\emph{Affine embeddings of homogeneous spaces},
in \emph{Surveys in Geometry and Number Theory}, N.~Young (ed.),
London Mathematical Society Lecture Note Series, vol.~338,
Cambridge University Press, Cambridge, 2007, pp.~1--51.
\href{https://doi.org/10.1017/CBO9780511721472.002}{doi:10.1017/CBO9780511721472.002}.

\bibitem{AFKKZ}
I.~Arzhantsev, H.~Flenner, S.~Kaliman, F.~Kutzschebauch, and M.~Zaidenberg,
\emph{Flexible varieties and automorphism groups},
Duke Math. J. \textbf{162} (2013), 767--823.
\href{https://doi.org/10.1215/00127094-2080132}{doi:10.1215/00127094-2080132}.

\bibitem{Borel}
A.~Borel,
\emph{Linear Algebraic Groups}, 2nd ed.,
Graduate Texts in Mathematics, vol.~126, Springer-Verlag, New York, 1991.

\bibitem{DonzelliDvorskyKaliman}
F.~Donzelli, A.~Dvorsky, and S.~Kaliman,
\emph{Algebraic density property of homogeneous spaces},
Transform. Groups \textbf{15} (2010), 551--576.
\href{https://doi.org/10.1007/s00031-010-9091-8}{doi:10.1007/s00031-010-9091-8}.

\bibitem{DvorskyRegular}
A.~Dvorsky,
\emph{Hamiltonian tameness and relative density on regular semisimple adjoint orbits},
manuscript (2026).

\bibitem{Forstneric}
F.~Forstneri\v c,
\emph{Stein Manifolds and Holomorphic Mappings: The Homotopy Principle in Complex Analysis},
2nd ed., Ergebnisse der Mathematik und ihrer Grenzgebiete, vol.~56,
Springer, Cham, 2017.
\href{https://doi.org/10.1007/978-3-319-61058-0}{doi:10.1007/978-3-319-61058-0}.

\bibitem{Humphreys}
J.~E. Humphreys,
\emph{Introduction to Lie Algebras and Representation Theory},
Graduate Texts in Mathematics, vol.~9, Springer-Verlag, New York, 1972.

\bibitem{KalimanKutzschebauch}
S.~Kaliman and F.~Kutzschebauch,
\emph{Algebraic (volume) density property for affine homogeneous spaces},
Math. Ann. \textbf{367} (2017), 1311--1332.
\href{https://doi.org/10.1007/s00208-016-1451-9}{doi:10.1007/s00208-016-1451-9}.

\bibitem{KnopRohrle}
F.~Knop and G.~R\"ohrle,
\emph{Spherical subgroups in simple algebraic groups},
Compos. Math. \textbf{151} (2015), 1288--1308.
\href{https://doi.org/10.1112/S0010437X1400791X}{doi:10.1112/S0010437X1400791X}.

\bibitem{Kraemer}
M.~Kr\"amer,
\emph{Sph\"arische Untergruppen in kompakten zusammenh\"angenden Liegruppen},
Compos. Math. \textbf{38} (1979), 129--153.

\bibitem{Losev}
I.~V. Losev,
\emph{Embeddings of homogeneous spaces into irreducible modules},
J. Algebra \textbf{322} (2009), 2621--2630.
\href{https://arxiv.org/abs/math/0606387}{arXiv:math/0606387}.

\bibitem{RosayRudin}
J.-P. Rosay and W.~Rudin,
\emph{Holomorphic maps from $\C^n$ to $\C^n$},
Trans. Amer. Math. Soc. \textbf{310} (1988), 47--86.
\href{https://doi.org/10.2307/2001110}{doi:10.2307/2001110}.

\bibitem{Stacks}
The Stacks Project Authors,
\emph{The Stacks Project}, Theorem~29.23.3, Tag~054K (Chevalley's theorem).
\url{https://stacks.math.columbia.edu/tag/054K}.

\bibitem{Varolin}
D.~Varolin,
\emph{The density property for complex manifolds and geometric structures},
J. Geom. Anal. \textbf{11} (2001), 135--160.
\href{https://doi.org/10.1007/BF02921959}{doi:10.1007/BF02921959}.

\bibitem{WinkelmannStein}
J.~Winkelmann,
\emph{Tame discrete subsets in Stein manifolds},
J. Aust. Math. Soc. \textbf{107} (2019), 110--132.
\href{https://doi.org/10.1017/S1446788718000241}{doi:10.1017/S1446788718000241}.

\bibitem{WinkelmannGroups}
J.~Winkelmann,
\emph{Tame discrete sets in algebraic groups},
Transform. Groups \textbf{26} (2021), 1487--1519.
\href{https://doi.org/10.1007/s00031-020-09608-x}{doi:10.1007/s00031-020-09608-x}.

\bibitem{WinkelmannRR}
J.~Winkelmann,
\emph{Rosay--Rudin-Spaces, Tame Sets and Danielewski Surfaces},
preprint (2023), \href{https://arxiv.org/abs/2309.07678}{arXiv:2309.07678}.

\end{thebibliography}
\end{document}